\RequirePackage[l2tabu,orthodox]{nag}		
\documentclass[reqno]{amsart}		
\usepackage[margin=1.5in,bottom=1.25in]{geometry}		

\usepackage{amsmath}		
\usepackage{amssymb}		
\usepackage{amsfonts}		
\usepackage{amsthm}		
\usepackage[foot]{amsaddr}		

\usepackage{mathtools}		

\mathtoolsset{%
}

\usepackage[utf8]{inputenc}		
\usepackage[T1]{fontenc}		

\usepackage[proportional,tabular,lining,sf,mono=false]{libertine}

\usepackage{dsfont}		

\usepackage[
cal=cm,
]
{mathalfa}

\usepackage{acronym}		
\renewcommand*{\aclabelfont}[1]{\acsfont{#1}}		
\newcommand{\acdef}[1]{\textit{\acl{#1}} \textup{(\acs{#1})}\acused{#1}}		

\usepackage[labelfont={bf,small},labelsep=colon,font=small]{caption}	
\usepackage[dvipsnames,svgnames]{xcolor}		
\colorlet{MyRed}{Crimson!90!Black}
\colorlet{MyBlue}{MediumBlue!90!Black}
\colorlet{MyGreen}{DarkGreen!80!Black}

\newcommand{\afterhead}{.}		
\newcommand{\para}[1]{\medskip\paragraph{\textbf{#1\afterhead}}}

\usepackage{pifont}

\usepackage{subcaption}		
\usepackage{tikz}		
\usetikzlibrary{calc,patterns,positioning}

\usepackage{booktabs}		
\usepackage[inline,shortlabels]{enumitem}		
\setenumerate{itemsep=\smallskipamount,topsep=\medskipamount}
\setitemize{itemsep=\smallskipamount,topsep=\medskipamount}

\usepackage[kerning=true]{microtype}		

\usepackage{xspace}		

\usepackage[numbers,sort&compress]{natbib}		

\bibpunct[, ]{[}{]}{,}{n}{,}{,}

\usepackage{hyperref}
\hypersetup{
colorlinks=true,
linktocpage=true,
pdfstartview=FitH,
breaklinks=true,
pdfpagemode=UseNone,
pageanchor=true,
pdfpagemode=UseOutlines,
plainpages=false,
bookmarksnumbered,
bookmarksopen=false,
bookmarksopenlevel=1,
hypertexnames=true,
pdfhighlight=/O,
urlcolor=MyBlue,linkcolor=MyBlue,citecolor=MyBlue,	
pdftitle={},
pdfauthor={},
pdfsubject={},
pdfkeywords={},
pdfcreator={pdfLaTeX},
pdfproducer={LaTeX with hyperref}
}

\usepackage[linesnumbered,ruled,vlined]{algorithm2e}

\SetCommentSty{mycommfont} 

\usepackage[sort&compress,capitalize,nameinlink]{cleveref}		
\crefname{algorithm}{Algorithm}{Algorithms}
\crefname{equation}{Eq.}{Eqs.}

\usepackage{thmtools}		
\usepackage{thm-restate}		

\theoremstyle{plain}
\newtheorem{theorem}{Theorem}		
\newtheorem{lemma}{Lemma}		

\newtheorem*{corollary*}{Corollary}		

\theoremstyle{definition}
\newtheorem{definition}{Definition}		
\newtheorem{example}{Example}		

\newtheorem*{definition*}{Definition}		
\newtheorem*{assumption*}{Assumptions}		
\newtheorem*{example*}{Example}		

\theoremstyle{remark}
\newtheorem{remark}{Remark}		
\newtheorem*{remark*}{Remark}		

\def\endenv{\hfill\raisebox{1pt}{\P}\smallskip}

\usepackage[suppress]{color-edits}		
\newcommand{\debug}[1]{#1}		

\newcommand{\newmacro}[2]{\newcommand{#1}{\debug{#2}}}		
\newcommand{\newop}[2]{\DeclareMathOperator{#1}{\debug{#2}}}		

\DeclarePairedDelimiter{\braces}{\{}{\}}		
\DeclarePairedDelimiter{\bracks}{[}{]}		
\DeclarePairedDelimiter{\parens}{(}{)}		

\DeclarePairedDelimiter{\abs}{\lvert}{\rvert}		

\DeclarePairedDelimiterX{\setdef}[2]{\{}{\}}{#1:#2}		
\DeclarePairedDelimiterXPP{\exclude}[1]{\mathopen{}\setminus}{\{}{\}}{}{#1}

\newcommand{\N}{\mathbb{N}}		
\newcommand{\R}{\mathbb{R}}		
\newcommand{\C}{\mathbb{C}}		

\DeclareMathOperator*{\argmax}{arg\,max}		
\DeclareMathOperator*{\argmin}{arg\,min}		

\DeclareMathOperator{\bigoh}{\mathcal{O}}		
\DeclarePairedDelimiterXPP{\bigof}[1]{\bigoh}{(}{)}{}{#1}		
\DeclareMathOperator{\conv}{conv}		
\DeclareMathOperator{\dist}{dist}		
\DeclareMathOperator{\dom}{dom}		
\DeclareMathOperator{\ess}{ess}		
\DeclareMathOperator{\relint}{ri}		
\DeclareMathOperator{\tr}{tr}		

\newcommand{\cf}{cf.\xspace}		
\newcommand{\ie}{i.e.,\xspace}		

\newcommand{\textpar}[1]{\textup(#1\textup)}		

\newcommand{\alt}[1]{#1'}		
\newcommand{\altalt}[1]{#1''}		
\newcommand{\avg}[1]{\bar#1}		

\newmacro{\dd}{\:d}		

\newcommand{\insum}{\sum\nolimits}		

\newmacro{\const}{C}		
\newmacro{\constalt}{B}		
\newmacro{\constdual}{\const_{\texttt{dual}}}		
\newmacro{\conststoch}{\constalt}		

\newmacro{\multi}{\texttt{H}}		

\newmacro{\coef}{\lambda}		
\newmacro{\param}{\theta}		
\newmacro{\params}{\Theta}		

\newmacro{\pexp}{r}		
\newmacro{\qexp}{q}		
\newmacro{\rexp}{r}		

\newmacro{\beforestart}{0}		
\newmacro{\start}{1}		
\newmacro{\afterstart}{2}		
\newmacro{\running}{\start,\afterstart,\dotsc\,}		

\newmacro{\run}{t}		
\newmacro{\runprev}{\run-1}		
\newmacro{\runalt}{s}		
\newmacro{\runaltalt}{\alt \run}		
\newmacro{\nRuns}{T}		
\newmacro{\runs}{\mathcal{\nRuns}}		

\newmacro{\state}{X}		
\newmacro{\stateavg}{\bar\state}		
\newmacro{\statealt}{\alt\state}		
\newmacro{\query}{x}		
\newmacro{\out}{\bar\query}		

\newcommand{\new}[1]{#1^{+}}		

\newcommand{\init}[1][\state]{\debug{#1}_{\start}}		

\newcommand{\iter}[1][\state]{\debug{#1}_{\runalt}}		
\newcommand{\iterlead}[1][\state]{\debug{#1}_{\runalt+1/2}}		

\newcommand{\prev}[1][\state]{\debug{#1}_{\run-1}}		
\newcommand{\curr}[1][\state]{\debug{#1}_{\run}}		
\renewcommand{\next}[1][\state]{\debug{#1}_{\run+1}}		

\newcommand{\lead}[1][\state]{\debug{#1}_{\run+1/2}}		

\newcommand{\last}[1][\state]{\debug{#1}_{\nRuns}}		
\newcommand{\lastlead}[1][\state]{\debug{#1}_{\nRuns+1/2}}		

\newop{\Eq}{Eq}		
\newop{\Nash}{NE}		

\newcommand{\gap}{\Delta}
\newop{\brep}{br}		
\newmacro{\reg}{\mathcal{R}}		
\newop{\regstoch}{\tilde{\reg}}		
\newop{\preg}{Reg}		

\newop{\val}{val}		
\newmacro{\play}{i}		
\newmacro{\playalt}{j}		
\newmacro{\playaltalt}{k}		
\newmacro{\nPlayers}{N}		
\newmacro{\players}{\mathcal{\nPlayers}}		

\newmacro{\pure}{\alpha}		
\newmacro{\purealt}{\beta}		
\newmacro{\purealtalt}{\gamma}		
\newmacro{\nPures}{n}		
\newmacro{\pures}{\mathcal{A}}		

\newmacro{\strat}{p}		
\newmacro{\stratalt}{\alt\strat}		
\newmacro{\strataltalt}{\altalt\strat}		
\newmacro{\strats}{\mathcal{X}}		
\newmacro{\intstrats}{\strats^{\circle}}		

\newmacro{\pay}{u}		
\newmacro{\payv}{v}		
\newmacro{\loss}{\ell}
\newmacro{\lossv}{\sample}
\newmacro{\pot}{\Phi}		
\newmacro{\meanpot}{\pot}		

\newmacro{\game}{\Gamma}		
\newmacro{\meangame}{\game}		
\newmacro{\gameall}{\game(\players,\points,\loss)}		

\newmacro{\fingame}{\Gamma}		
\newmacro{\fingameall}{\Gamma(\players,\pures,\pay)}		

\newmacro{\gmat}{g}		
\newmacro{\gdist}{\dist_{\gmat}}
\newmacro{\mfld}{M}		
\newmacro{\form}{\omega}		

\newmacro{\tvec}{z}		
\newmacro{\uvec}{u}		

\newmacro{\ball}{\mathbb{B}}		
\newmacro{\sphere}{\mathbb{S}}		

\newmacro{\vecspace}{\mathcal{V}}		
\newmacro{\subspace}{\mathcal{W}}		

\newmacro{\bvec}{e}		
\newmacro{\bvecs}{\mathcal{E}}		

\newmacro{\coord}{i}		
\newmacro{\coordalt}{\coord^{\prime}}		
\newmacro{\coordaltalt}{\coordalt^{\prime}}		
\newmacro{\nCoords}{d}		
\newmacro{\dims}{\nCoords}		
\newmacro{\vdim}{\nCoords}		

\newmacro{\pspace}{\vecspace}		
\newmacro{\dspace}{\vecspace^{\ast}}		

\newmacro{\pstate}{z}		

\newmacro{\dstate}{Y}		
\newmacro{\dvec}{{v}}		
\newmacro{\disdstate}{\hat\dstate}		

\newmacro{\anchor}{Z}
\newmacro{\ptest}{\tilde{\pstate}}		
\newmacro{\test}{\tilde{\state}}		
\newmacro{\dtest}{\tilde{\drecom}}		
\newmacro{\testsignal}{\tilde{\signal}}		
\newmacro{\precom}{\pstate}	
\newmacro{\drecom}{\dstate}		

\newmacro{\dpoint}{y}		

\newmacro{\dpointave}{\bar\dpoint}		
\newmacro{\dpointalt}{W}		
\newmacro{\dpointaltalt}{\altalt\dpoint}		
\newmacro{\dpoints}{\mathcal{Y}}		

\newmacro{\mat}{\mathbf{M}}		
\newmacro{\hmat}{\mathbf{H}}		

\newmacro{\ones}{\mathbf{1}}		
\newmacro{\eye}{\mathbf{I}}		
\newmacro{\zer}{\mathbf{0}}		

\newcommand{\mgeq}{\succcurlyeq}		

\DeclarePairedDelimiter{\norm}{\lVert}{\rVert}		
\DeclarePairedDelimiterXPP{\dnormdef}[1]{}{\lVert}{\rVert}{_{\ast}}{#1}
\DeclarePairedDelimiterXPP{\dnorm}[1]{}{\lVert}{\rVert}{_{\ast}}{#1}

\DeclarePairedDelimiterXPP{\onenorm}[1]{}{\lVert}{\rVert}{_{1}}{#1}		
\DeclarePairedDelimiterXPP{\twonorm}[1]{}{\lVert}{\rVert}{_{2}}{#1}		
\DeclarePairedDelimiterXPP{\supnorm}[1]{}{\lVert}{\rVert}{_{\infty}}{#1}		

\DeclarePairedDelimiter{\altnorm}{\lVert}{\rVert'}		
\DeclarePairedDelimiterXPP{\altdnorm}[1]{}{\lVert}{\rVert'}{_{\ast}}{#1}

\DeclarePairedDelimiterX{\braket}[2]{\langle}{\rangle}{#1\mathopen{},\mathopen{}#2}

\DeclarePairedDelimiterX{\inner}[2]{\langle}{\rangle}{#1,#2}		
\newmacro{\cartprod}{\bigtimes}

\newcommand{\defeq}{\coloneqq}		
\newcommand{\eqdef}{\eqqcolon}		

\newcommand{\from}{\colon}		

\newop{\Opt}{Opt}		
\newop{\Sol}{Sol}		
\newop{\orcl}{\mathsf{G}}		

\newmacro{\obj}{f}		
\newmacro{\objalt}{\alt \obj}		
\newmacro{\sobj}{F}		
\newmacro{\func}{\textsl{g}}

\newmacro{\gvec}{g}		
\newmacro{\oper}{A}		
\newmacro{\vecfield}{v}		
\newmacro{\seed}{\sample}		

\newcommand{\sol}[1][\point]{#1^{\ast}}		
\newcommand{\sols}{\sol[\points]}		

\newmacro{\gbound}{G}		
\newmacro{\lips}{L}

\newmacro{\strong}{\kappa}		

\newmacro{\cvx}{\mathcal{K}}		
\newmacro{\subd}{\partial}		
\newmacro{\subsel}{\nabla}		

\newmacro{\minmax}{L}		

\newmacro{\minvar}{\theta}		
\newmacro{\minvaralt}{\alt\minvar}		
\newmacro{\minvars}{\Theta}		

\newmacro{\maxvar}{\phi}		
\newmacro{\maxvaralt}{\alt\maxvar}		
\newmacro{\maxvars}{\Phi}		

\newmacro{\hreg}{h}		
\newmacro{\proxdom}{\points_{\hreg}}		

\newmacro{\breg}{D}		
\newmacro{\mprox}{P}		

\newmacro{\fench}{F}		
\newmacro{\mirror}{Q}		

\newmacro{\hstr}{K_{\hreg}}		
\newmacro{\bregdiam}{B_{\hreg}}		
\newmacro{\range}{R_{\hreg}}		

\DeclarePairedDelimiterXPP{\proxof}[2]{\mprox_{#1}}{(}{)}{}{#2}		
\DeclarePairedDelimiterXPP{\bregof}[2]{\breg}{(}{)}{}{#1,#2}		
\DeclarePairedDelimiterXPP{\fenchof}[2]{\fench}{(}{)}{}{#1,#2}		

\newmacro{\zone}{\mathbb{D}}		

\newop{\Eucl}{\Pi}		
\newop{\logit}{\Lambda}		
\newmacro{\subgrad}{W}

\newmacro{\point}{x}		
\newmacro{\pointalt}{\alt\point}		
\newmacro{\pointaltalt}{\altalt\point}		
\newmacro{\points}{\mathcal{X}}		
\newmacro{\intpoints}{\relint\points}		

\newmacro{\base}{\point^{\ast}}		
\newmacro{\basealt}{u^{\ast}}		

\newmacro{\real}{a}
\newmacro{\realalt}{\alt\real}
\newmacro{\realaltalt}{\altalt \real}

\newmacro{\open}{\mathcal{U}}		
\newmacro{\closed}{\mathcal{C}}		
\newmacro{\cpt}{\mathcal{K}}		
\newmacro{\nhd}{\mathcal{U}}		

\newop{\ex}{\mathbb{E}}		
\newop{\prob}{\mathbb{P}}		
\newop{\Var}{Var}		
\newop{\simplex}{\Delta}		

\providecommand\given{}		

\DeclarePairedDelimiterXPP{\exof}[1]{\ex}{[}{]}{}{
	\renewcommand\given{\nonscript\,\delimsize\vert\nonscript\,\mathopen{}} #1}

\DeclarePairedDelimiterXPP{\probof}[1]{\prob}{(}{)}{}{
	\renewcommand\given{\nonscript\:\delimsize\vert\nonscript\:\mathopen{}} #1}

\newmacro{\sample}{\omega}		
\newmacro{\samples}{\Omega}		

\newmacro{\filter}{\mathcal{F}}		
\newmacro{\probspace}{(\samples,\filter,\prob)}		

\newmacro{\event}{E}       
\newmacro{\eventalt}{H}       
\newmacro{\mean}{\mu}		
\newmacro{\sdev}{\sigma}		
\newmacro{\variance}{\sdev^{2}}		

\newmacro{\step}{\alpha}		
\newmacro{\stepalt}{\gamma}		
\newmacro{\stepaltalt}{\theta}		
\newmacro{\stepada}{\theta}		

\newmacro{\learn}{\eta}		
\newmacro{\dstep}{\psi}		

\newmacro{\proper}{\tau}		

\newmacro{\signal}{\gvec}		
\newmacro{\altsignal}{\bar\signal}		
\newmacro{\error}{Z}		
\newmacro{\noise}{U}		
\newmacro{\bias}{b}		
\newmacro{\brown}{W}		

\newmacro{\serror}{\theta}		
\newmacro{\snoise}{\xi}		
\newmacro{\sbias}{\psi}		

\newmacro{\sbound}{M}		
\newmacro{\bbound}{B}		
\newmacro{\noisepar}{\sdev}		
\newmacro{\noisevar}{\texttt{var}}		

\newmacro{\denom}{\real^{2}}
\newmacro{\denomsqrt}{\real}
\newmacro{\numer}{b}

\newcommand{\accelegrad}{\debug{\textsc{AcceleGrad}}\xspace}
\newcommand{\unixgrad}{\debug{\textsc{UnixGrad}}\xspace}
\newcommand{\undergrad}{\debug{\textsc{UnderGrad}}\xspace}

\newmacro{\cost}{c}

\newmacro{\noiseave}{\tilde{\noise}}  
\newmacro{\noisetest}{\tilde{\noise}}  

\newmacro{\mindiff}{B}
\newmacro{\diff}{\xi}

\colorlet{KAcolor}{DarkMagenta}
\addauthor[\textbf{Kim}]{KA}{KAcolor}

\colorlet{DQVcolor}{DarkGreen}
\addauthor[\textbf{Quan}]{DQV}{DQVcolor}

\colorlet{VCcolor}{Crimson}
\addauthor[\textbf{Volkan}]{VC}{VCcolor}

\colorlet{KLcolor}{Chocolate!80!DarkRed}
\addauthor[\textbf{Kfir}]{KL}{KLcolor}

\colorlet{PMcolor}{Blue}
\addauthor[\textbf{Pan}]{PM}{PMcolor}

\newmacro{\resource}{s}
\newmacro{\nResources}{d}
\newmacro{\resources}{\mathcal{S}}
\newmacro{\inflow}{\rho}
\newmacro{\load}{x}

\newmacro{\spectron}{\mathcal{D}}
\newmacro{\nInput}{M}
\newmacro{\nOutput}{N}
\newmacro{\chanmat}{\mathbf{H}}
\newmacro{\covmat}{\mathbf{X}}
\newmacro{\bx}{\mathbf{x}}
\newmacro{\by}{\mathbf{y}}
\newmacro{\bz}{\mathbf{z}}
\newmacro{\capa}{R}
\newmacro{\power}{P}

\newmacro{\maxsel}{m}

\newmacro{\shape}{\chi}
\newmacro{\diampoints}{\norm{\points}}

\newmacro{\underconst}{\const}
\newmacro{\hconst}{\const_{\hreg}}
\newmacro{\slow}{\textrm{slow}}
\newmacro{\fast}{\textrm{fast}}

\begin{document}


\title
[A Universal First-Order Method with Almost Dimension-Free Guarantees]
{A Universal Black-Box Optimization Method with\\
Almost Dimension-Free Convergence Rate Guarantees}

\author
[K.~Antonakopoulos]
{Kimon Antonakopoulos$^{\ast,\diamond}$}
\address{$^{\sharp}$\,%
Laboratory for Information and Inference Systems, IEL STI EPFL, Lausanne, Switzerland.}
\email{kimon.antonakopoulos@epfl.ch}

\author
[D.~Q.~Vu]
{Dong Quan Vu$^{\sharp,\diamond}$}
\address{$^{\diamond}$\,%
Safran Tech, Magny-Les-Hameaux, France.}
\email{dong-quan.vu@safrangroup.com}

\author
[V.~Cevher]
{Volkan Cevher$^{\ast}$}
\email{volkan.cevher@epfl.ch}

\author
[K.~Y.~Levy]
{\\Kfir Yehuda Levy$^{\dag}$}
\address{$^{\dag}$\,%
Technion, Haifa, Israel.}
\email{kfirylevy@technion.ac.il}

\author
[P.~Mertikopoulos]
{Panayotis Mertikopoulos$^{\ddag,\S}$}
\address{$^{\ddag}$\,%
Univ. Grenoble Alpes, CNRS, Inria, Grenoble INP, LIG, 38000 Grenoble, France.}
\address{$^{\S}$\,%
Criteo AI Lab.}
\email{panayotis.mertikopoulos@imag.fr}


\thanks{This work was done when KA and DQV were affiliated with Univ. Grenoble Alpes, CNRS, Inria, Grenoble INP, LIG, 38000 Grenoble, France.}


\newacro{LHS}{left-hand side}
\newacro{RHS}{right-hand side}
\newacro{iid}[i.i.d.]{independent and identically distributed}
\newacro{lsc}[l.s.c.]{lower semi-continuous}

\newacro{undergrad}[\undergrad]{universal dual extrapolation with reweighted gradients}
\newacro{SFO}{stochastic first-order oracle}
\newacro{EGD}{entropic gradient descent}
\newacro{EG}{extra-gradient}
\newacro{MD}{mirror descent}
\newacro{MP}{mirror-prox}
\newacro{DE}{dual extrapolation}
\newacro{VI}{variational inequality}
\newacroplural{VI}[VIs]{variational inequalities}
\newacro{ACSA}[AC-SA]{accelerated stochastic approximation}
\newacro{UPGD}{universal primal gradient descent}

\begin{abstract}

Universal methods for optimization are designed to achieve theoretically optimal convergence rates without any prior knowledge of the problem's regularity parameters or the accurarcy of the gradient oracle employed by the optimizer.
In this regard, existing state-of-the-art algorithms achieve an $\bigoh(1/\nRuns^{2})$ value convergence rate in Lipschitz smooth problems with a perfect gradient oracle, and an $\bigoh(1/\sqrt{\nRuns})$ convergence rate when the underlying problem is non-smooth and/or the gradient oracle is stochastic.
On the downside, these methods do not take into account the problem's dimensionality, and this can have a catastrophic impact on the achieved convergence rate, in both theory and practice.
Our paper aims to bridge this gap by providing a scalable universal gradient method \textendash\ dubbed \acs{undergrad} \textendash\ whose oracle complexity is almost dimension-free in problems with a favorable geometry (like the simplex, linearly constrained semidefinite programs and combinatorial bandits), while retaining the order-optimal dependence on $\nRuns$ described above.
These ``best-of-both-worlds'' results are achieved via a primal-dual update scheme inspired by the dual exploration method for variational inequalities.
\end{abstract}

\maketitle
\allowdisplaybreaks		
\acresetall		

\section{Introduction}
\label{sec:intro}


The analysis of first-order methods for convex minimization typically revolves around the following regularity conditions:
\begin{enumerate*}
[\itshape a\upshape)]
\item
Lipschitz continuity of a problem's objective function
and\,/\,or
\item
Lipschitz continuity of the objective's gradients.
\end{enumerate*}
Depending on these conditions and the quality of the gradient oracle available to the optimizer, the optimal convergence rates that can be obtained by an iterative first-order algorithm after $\nRuns$ oracle queries are:
\begin{enumerate}
\item
$\bigof[\big]{\diampoints\sqrt{(\gbound^{2} + \sdev^{2})/\nRuns}}$
if the problem's objective is $\gbound$-Lipschitz continuous and the oracle's variance is $\sdev$.
\item
$\bigof[\big]{\lips\diampoints^{2}/\nRuns^{2} + \sdev\diampoints/\sqrt{\nRuns}}$
if the objective is $\lips$-Lipschitz smooth.
\end{enumerate}
[In both cases, $\diampoints \defeq \sup_{\point,\pointalt\in\points} \norm{\pointalt - \point}$ denotes the diameter of the problem's domain $\points \subseteq \R^{\vdim}$;
for an in-depth treatment, see \cite{Nes04,Bub15} and references therein.]

This stark separation of black-box guarantees has led to an intense search for \emph{universal} methods that are capable of interpolating smoothly between these rates without any prior knowledge of the problem's regularity properties or the oracle's noise profile.
As far as we are aware, the first algorithm with order-optimal rate guarantees for unconstrained problems and no knowledge of the problem's smoothness parameters was the \accelegrad proposal of \citet{LYC18}.
Subsequently, in the context of \emph{constrained} convex problems (the focus of our work), \citet{KLBC19} combined the \acl{EG}\,/\,\acl{MP} algorithmic template of \citet{Kor76} and \citet{Nem04} with an ``iterate averaging'' scheme introduced by \citet{Cut19} to change the query structure of the base algorithm and make it more amenable to acceleration.
In this way, \citet{KLBC19} obtained a universal \acl{EG} algorithm \textendash\ dubbed \unixgrad \textendash\ which interpolates between the optimal rates mentioned above, without requiring any tuning.

\para{Our contributions}

The starting point of our paper is the observation that, even though the rates in question are optimal in $\nRuns$, they may be highly supoptimal in $\vdim$, the problem's dimensionality.
For example, if the noise in the oracle has unit variance, $\sdev$ would scale as $\bigoh(\sqrt{\vdim})$;
this represents a hidden dependence on $\vdim$ which could have a catastrophic impact on the method's actual convergence rate.
Likewise, in problems with a favorable geometry (like the $L^{1}$-ball, trace-constrained semidefinite programs, combinatorial bandits, etc.), methods based on the \acl{MD} \cite{NY83} and \acl{MP} \cite{Nem04} templates can achieve rates with a \emph{logarithmic} (instead of polynomial) dependence on $\vdim$.

Importantly, the \unixgrad algorithm of \citet{KLBC19} is itself based on the \acl{MP} blueprint, so it would seem ideally suited to achieve convergence rates that are simultaneously optimal in $\nRuns$ and $\vdim$.
However, the method's guarantees depend crucially on the \emph{Bregman diameter} of the problem's domain, a quantity which becomes infinite when the method is used with a regularization setup that can lead to almost dimension-free guarantees.
This would seem to suggest that universality comes at the cost of scalability, leading to the following open question:
\begin{center}
\itshape
Is it possible to achieve almost dimension-free convergence rates\\
while retaining an order-optimal dependence on $\nRuns$?
\end{center}

In this paper, we develop a novel adaptive algorithm, which we call \acdef{undergrad}, and which provides a positive answer to this question.
Specifically, the value convergence rate of \ac{undergrad} scales in terms of $\gbound$, $\sdev$, $\lips$ and $\nRuns$ as:
\begin{enumerate}
\item
$\bigof[\big]{\sqrt{\range(\gbound^{2} + \sdev^{2})/\nRuns}}$ in non-smooth problems.
\item
$\bigof[\big]{\range\lips/\nRuns^{2} + \sdev\sqrt{\range/\nRuns}}$ in smooth problems.
\end{enumerate}
In the above, the method's ``range parameter'' $\range$ scales as $\bigoh(\diampoints^{2})$ in the worst case and as $\bigoh(\log\vdim)$ in problems with a favorable geometry \textendash\ that is, in problems where it is possible to attain almost dimension-free convergence rates \citep{Nes04,Bub15}.
In this regard, \ac{undergrad} seems to be the first method in the literature that concurrently achieves order-optimal rates in both $\nRuns$ and $\vdim$, without any prior knowledge on the problem's level of smoothness.

To achieve this result, the \ac{undergrad} algorithm combines the following basic ingredients:
\begin{enumerate}
\item
A modified version of the \acl{DE} method of \citet{Nes07} for solving \aclp{VI}.
\item
A gradient ``reweighting'' scheme that allows gradients to enter the algorithm with \emph{increasing} weights.
\item
An iterate averaging scheme in the spirit of \citet{Cut19} which allows us to obtain an accelerated rate of convergence by means of an online-to-batch conversion.
\end{enumerate}
The glue that holds these elements together is an adaptive learning rate inspired by \citet{RS13-COLT,RS13-NIPS} which automatically rescales aggregated gradients by 
\begin{enumerate*}
[\itshape a\upshape)]
\item
a small, constant amount when the method approaches a solution where gradient differences vanish (as in the smooth, deterministic case);
and
\item
a factor of $\bigoh(\sqrt{\nRuns})$ otherwise (thus providing the desired interpolation between smooth and non-smooth problems).
\end{enumerate*}
In so doing, the proposed policy achieves the correct step-size scaling and achieves the desired optimal rates.

\para{Related work}

The term ``universality'' was coined by \citet{Nes15} whose \acdef{UPGD} algorithm interpolates between the $\bigoh(1/\nRuns^{2})$ and $\bigoh(1/\sqrt{\nRuns})$ rates for smooth and non-smooth problems respectively (assuming access to noiseless gradients in both cases).
On the downside, \ac{UPGD} relies on an Armijo-like line search to interpolate between smooth and non-smooth objectives, so it is not applicable to stochastic environments.

A partial work-around to this issue was achieved by the \ac{ACSA} algorithm of \citet{Lan12} which uses a \acl{MD} template and guarantees order-optimal rates for both noisy \emph{and} noiseless oracles.
However, to attain these rates, the \ac{ACSA} algorithm requires a precise estimate of the smoothness modulus of the problem's objective, so it is not universal in this respect.
Subsequent works on the topic have focused on attaining universal guarantees for composite problems \citep{JRGS20}, non-convex objectives \cite{LO19,WWB19}, preconditioned methods \citep{JRGS20,ENV21}, non-Lipschitz settings \citep{ABM19,AM21,ABM21}, specific applications \citep{VAM21}, or \aclp{VI} / min-max problems \citep{BL19,ABM21,HAM21,APKM+21}.

Of the generalist works above, some employ a Bregman regularization setup \citep{AM21,BL19}, but the guarantees obtained therein either fall short of an accelerated $\bigoh(1/\nRuns^{2})$ convergence rate for Lipschitz smooth problems, or they depend on the problem's Bregman diameter \textendash\ so they cannot be associated with a Bregman setup leading to almost dimension-free convergence rate guarantees.
To the best of our knowledge, \ac{undergrad} is the first method that manages to combine the ``best of both worlds'' in terms of universality with respect to $\nRuns$ and scalability with respect to $\vdim$.


%
%

\section{Preliminaries}
\label{sec:prelims}


\subsection{Notation and basic definitions}

Let $\vecspace$ be a $\vdim$-dimensional space with norm $\norm{\cdot}$.
In what follows, we will write
$\dpoints \equiv \dspace$ for the dual of $\vecspace$,
$\braket{\dpoint}{\point}$ for the pairing between $\dpoint\in\dpoints$ and $\point\in\vecspace$,
and
$\dnorm{\dpoint} \equiv \sup\setdef{\braket{\dpoint}{\point}}{\norm{\point} \leq 1}$ for the dual norm on $\dpoints$.
Given an extended-real-valued convex function $\obj\from\vecspace\to\R\cup\{\infty\}$, we will write
$\dom\obj \equiv \setdef{\point\in\vecspace}{\obj(\point) < \infty}$ for its \emph{effective domain}
and
\(
\subd\obj(\point)
	\equiv \setdef
		{\dpoint\in\dpoints}
		{\obj(\pointalt) - \obj(\point) - \braket{\dpoint}{\pointalt - \point}\geq 0\;\text{for all $\pointalt\in\vecspace$}}
\)
for the \emph{subdifferential} of $\obj$ at $\point\in\dom\obj$.
Any element $\gvec\in\subd\obj(\point)$ will be called a \emph{subgradient} of $\obj$ at $\point$, and we will write $\dom\subd\obj \equiv \setdef{\point\in\dom\obj}{\subd\obj \neq \varnothing}$ for the \emph{domain of subdifferentiability} of $\obj$.

\subsection{Problem setup and blanket assumptions}

The main focus of our paper is the solution of convex minimization problems of the form
\begin{equation}
\label{eq:opt}
\tag{Opt}
\begin{aligned}
\textrm{minimize}
	&\quad
	\obj(\point)
	\\
\textrm{subject to}
	&\quad
	\point\in\points
\end{aligned}
\end{equation}
where
$\points$ is a closed convex subset of $\vecspace$
and
$\obj\from\vecspace\to\R\cup\{\infty\}$ is a convex function with $\dom\obj = \dom\subd\obj = \points$.
To avoid trivialities, we will assume throughout that the solution set $\sols \defeq \argmin\obj$ of \eqref{eq:opt} is non-empty, and we will write $\sol$ for a generic minimizer of $\obj$.

Other than this blanket assumption, our main reqularity requirements for $\obj$ will be as follows:
\begin{enumerate}
\item
\emph{Lipschitz continuity:}
\begin{equation}
\label{eq:LC}
\tag{LC}
\abs{\obj(\pointalt) - \obj(\point)}
	\leq \gbound \norm{\pointalt - \point}
\end{equation}
for some $\gbound\geq0$ and for all $\point,\pointalt\in\points$.
\item
\emph{Lipschitz smoothness:}
\begin{equation}
\label{eq:LS}
\tag{LS}
\obj(\pointalt)
	\leq \obj(\point)
		+ \braket{\nabla\obj(\point)}{\pointalt - \point}
		+ \frac{\lips}{2} \norm{\pointalt - \point}^{2}
\end{equation}
for some $\lips\geq0$ and for all $\gvec\in\subd\obj(\point)$, $\point,\pointalt\in\points$.
\end{enumerate}

Since $\dom\subd\obj = \points$, the above requirements are respectively equivalent to assuming
that $\obj$ admits a selection of subgradients $\nabla\obj(\point) \in \subd\obj(\point)$ with the properties below:
\begin{enumerate}
\item
\emph{Bounded \textpar{sub}gradient selection:}
\begin{equation}
\label{eq:BG}
\tag{BG}
\dnorm{\nabla\obj(\point)}
	\leq \gbound
\end{equation}
for some $\gbound\geq0$ and for all $\point\in\points$.
\item
\emph{Lipschitz \textpar{sub}gradient selection:}
\begin{equation}
\label{eq:LG}
\tag{LG}
\dnorm{\nabla\obj(\pointalt) - \nabla\obj(\point)}
	\leq \lips \norm{\pointalt - \point}
\end{equation}
for some $\lips\geq0$ and for all $\point,\pointalt\in\points$.
\end{enumerate}
In the rest of our paper, we will assume that $\obj$ satisfies at least one of \eqref{eq:BG} or \eqref{eq:LG}.

\begin{remark}
For posterity, we note here that the requirement \eqref{eq:LG} \emph{does not} imply that $\subd\obj(\point)$ is a singleton.%
\footnote{Consider for example the case of $\obj(\point) = \point$ for $\point\in[0,1]$ and $\obj(\point) = \infty$ otherwise:
$\obj$ clearly satisfies \eqref{eq:BG}/\eqref{eq:LS}, even though its $\subd\obj(0)$ and $\subd\obj(1)$ are infinite sets.}
In any case, the directional derivative $\obj'(\point;\tvec) = d/dt |_{t=0} \obj(\point + t\tvec)$ of $\obj$ at $\point\in\points$ along $\tvec\in\vecspace$ exists and is equal to $\braket{\nabla\obj(\point)}{\tvec}$ for all vectors of the form $\tvec = \pointalt - \point$, $\pointalt\in\points$.
We will use this fact freely in the sequel.
\endenv
\end{remark}

\subsection{The oracle model}

To solve \eqref{eq:opt}, we will consider iterative methods and algorithms with access to a \acdef{SFO}, \ie a black-box device that returns a (possibly random) estimate of a subgradient of $\obj$ at the point at which it was queried.
Formally, following \citet{Nes04}, an \ac{SFO} for $\obj$ is a measurable function $\orcl\from\points\times\samples \to \dpoints$ such that
\begin{equation}
\label{eq:SFO}
\tag{SFO}
\exof{\orcl(\point;\sample)}
	= \nabla\obj(\point)
	\quad
	\text{for all $\point\in\points$}
\end{equation}
where
$\probspace$ is a complete probability space
and
$\nabla\obj(\point)$ is a selection of subgradients of $\obj$ as per \eqref{eq:BG}/\eqref{eq:LG}.
The oracle's statistical precision will then be measured by the associated \emph{noise level} $\sdev \defeq \ess\sup_{\sample,\point} \dnorm{\orcl(\point;\sample) - \nabla\obj(\point)}$ (assumed finite).
In particular, if $\sdev=0$, $\orcl$ will be called \emph{perfect} (or \emph{deterministic});
otherwise, $\orcl$ will be called \emph{noisy}.

In practice, the oracle is called repeatedly at a sequence of query points $\curr[\query]$ with a different random seed $\curr[\sample]$ drawn according to $\prob$ at each time.%
\footnote{In the sequel, $\run$ may take both integer and half-integer values.}
In this way, at the $\run$-th query to \eqref{eq:SFO}, the oracle $\orcl$ returns the gradient signal
\begin{equation}
\label{eq:signal}
\curr[\signal]
	= \orcl(\curr[\query];\curr[\sample])
	= \nabla\obj(\curr[\query]) + \curr[\noise]
\end{equation}
where $\curr[\noise]$ denotes the ``gradient noise'' of the oracle (obviously, $\curr[\noise] \equiv 0$ if the oracle is perfect).
For measurability purposes, we will write $\curr[\filter]$ for the history (adapted filtration) of $\curr[\query]$, so $\curr[\query]$ is $\curr[\filter]$-measurable (by definition) but $\curr[\sample]$, $\curr[\signal]$ and $\curr[\noise]$ are not.
In particular, conditioning on $\curr[\filter]$, we have $\exof{\curr[\signal] \given \curr[\filter]} = \nabla\obj(\curr[\query])$ and $\exof{\curr[\noise] \given \curr[\filter]} = 0$, justifying in this way the terminology ``gradient noise'' for $\curr[\noise]$.

\begin{remark}
The oracle model detailed above is not the only one possible, but it is very widely used in the analysis of parameter-agnostic and adaptive methods, \cf \cite{LYC18,KLBC19,WWB19,AM21} and references therein.
In view of this, we will not examine either finer or coarser assumptions for \eqref{eq:SFO}.
\endenv
\end{remark}

We close this section by noting that the best convergence rates that can be achieved by an iterative algorithm that outputs a candidate solution $\last[\out] \in \points$ after $\nRuns$ queries to \eqref{eq:SFO} are:%
\footnote{%
In general, the query and output points \textendash\ $\last[\query]$ and $\last[\out]$ respectively \textendash\ need not coincide, hence the different notation.
The only assumption for the rates provided below is that the output point $\last[\out]$ is an affine combination of $\init[\query],\init[\signal],\dotsc,\last[\query],\last[\signal]$ \citep{Nes04,Bub15}.}
\begin{enumerate}
\item
$\obj(\last[\out]) - \min\obj = \bigoh(1/\sqrt{\nRuns})$ if $\obj$ satisfies \eqref{eq:BG} and $\orcl$ is deterministic.

\item
$\obj(\last[\out]) - \min\obj = \bigoh(1/\nRuns^{2})$ if $\obj$ satisfies \eqref{eq:LG} and $\orcl$ is deterministic.

\item
$\exof{\obj(\last[\out]) - \min\obj} = \bigoh(1/\sqrt{\nRuns})$ if $\orcl$ is stochastic.
\end{enumerate}
In general, without finer assumptions on $\obj$ or $\orcl$, the dependence of these rates on $\nRuns$ cannot be improved  \citep{Nes04,Bub15};
we will revisit this issue several times in the sequel.

\section{Regularization, universality, and the curse of dimensionality}
\label{sec:base}

To set the stage for the analysis to come, we discuss below the properties of two algorithmic frameworks \textendash\ one non-adaptive, the other adaptive \textendash\ based on the \acl{MP} template \citep{Nem04}.
Our aim in doing this will be to set a baseline for the sequel as well as to explore the impact of the problem's dimensionality on the attained rates of convergence.

\subsection{Motivating examples}
\label{sec:examples}

As a first step, we present three archetypal problems to motivate and illustrate the general setup that follows.


\begin{example}
[Resource allocation]
\label{ex:simplex}
Consider a set of computing resources (GPUs in a cluster, servers in a computing grid, \dots) indexed by $\resource \in \resources = \{1,\dotsc,\nResources\}$.
Each resource is capable of serving a stream of computing demands that arrive at a rate of $\inflow$ units per time:
if the optimizer assigns a load of $\load_{\resource}\geq0$ to the $\resource$-th resource, the marginal cost incurred is $\cost_{\resource}(\load_{\resource})$ per unit served, where $\cost_{\resource}\from[0,\inflow]\to\R_{+}$ is the cost function of the $\resource$-th resource (assumed convex, differentiable, and increasing in $\load_{\resource}$).
Taking $\inflow=1$ for simplicity, the goal of the optimizer is to minimize the aggregate cost $\obj(\load) = \sum_{\resource=1}^{\nResources} \load_{\resource} \cost_{\resource}(\load_{\resource})$, leading to a convex minimization problem over the unit $\nResources$-dimensional simplex $\points = \simplex(\resources) = \setdef{\load\in\R_{+}^{\nResources}}{\sum_{\resource}\load_{\resource} = 1}$.
\endenv
\end{example}



\begin{example}
[Input covariance matrix optimization]
\label{ex:spectron}
Consider a Gaussian vector channel in the spirit of \cite{Tel99,YRBC04}:
the encoder controls the covariance matrix $\covmat = \exof{\bx\bx^{\dag}}$ of the Gaussian input signal $\bx \in \C^{\nInput}$ and seeks to maximize the transfer rate of the output signal $\by = \chanmat \bx + \bz$, where $\bz\in\C^{\nOutput}$ is the ambient noise in the channel and $\chanmat \in\C^{\nOutput\times\nInput}$ is the channel's transfer matrix.
By the Shannon\textendash Telatar formula \citep{Tel99}, this boils down to maximizing the capacity function
\begin{equation}
\label{eq:capa}
\capa(\covmat)
	= \exof*{\log\det\parens*{\eye + \chanmat \covmat \chanmat^{\dag}}}
\end{equation}
subject to the constraint $\tr(\covmat) \leq \power$, where $\power$ denotes the encoder's maximum input power and the expectation in \eqref{eq:capa} is taken over the statistics of the (possibly deterministic) matrix $\chanmat$.
Since $\capa$ is concave in $\covmat$ \citep{BV04,YRBC04}, we obtain a minimization problem of the form \eqref{eq:opt} over the spectrahedron
\(
\spectron
	= \setdef{\covmat \mgeq 0}{\tr(\covmat) \leq \power}.
\)
Since $\covmat$ is Hermitian, $\spectron$ can be seen as a convex body of $\R^{\vdim}$ where $\vdim = \nInput^{2}$;
in the optimization literature, this is sometimes referred to as the ``spectrahedron setup'' \citep{JNT11}.
\endenv
\end{example}



\begin{example}
[Combinatorial bandits]
\label{ex:comband}
In bandit linear optimization problems, the optimizer is given a finite set of $\nPures$ possible \emph{actions} $\pures \subseteq \{0,1\}^{\vdim}$, \ie each action $\pure\in\pures$ is a $\vdim$-dimensional binary vector indicating whether the $\coord$-th component is ``on'' or ``off''.
The optimizer then chooses an action $\pure\in\pures$ based on a mixed strategy $\strat\in\simplex(\pures)$ and incurs the mean loss
\begin{equation}
\loss(\strat;\lossv)
	= \exof*{\insum_{\pure\in\pures} \strat_{\pure} \braket{\pure}{\lossv}}
\end{equation}
where $\lossv$ is a random vector with values in $[0,1]^{\vdim}$ (but otherwise unknown distribution).
In many cases of interest \textendash\ such as slate recommendation and shortest-path problems \textendash\ the cardinality of $\pures$ is exponential in $\vdim$, so it is computationally prohibitive to state the resulting minimization problem in terms of $\strat$.
Instead, writing $\point_{\coord} = \sum_{\pure\in\pures} \strat_{\pure} \pure_{\coord}$ for the probability of the $\coord$-th component being ``on'' under $\strat$, the optimizer's objective can be rewritten more compactly as
\(
\obj(\point)
		= \exof*{\braket{\point}{\lossv}}
\)
with $\point$ constrained to lie on the $\vdim$-dimensional convex hull $\points = \conv(\pures)$ of $\pures$ in $\R^{\vdim}$.
In the literature on multi-armed bandits, this setup is known as a \emph{combinatorial bandit};
for an in-depth treatment, see \citep{LS20,CBL12,GLLO07,CBL06} and the many references cited therein.
\endenv
\end{example}


\Cref{ex:simplex,ex:spectron,ex:comband} all suffer from the ``curse of dimensionality'':
for instance, the dimensionality of a vector Gaussian channel with $\nInput = 256$ input entries is $\vdim \approx 6.5\times 10^{4}$, while a combinatorial bandit for recommendation systems may have upwards of several million arms.
Nonetheless, these examples also share a number of geometric properties that make it possible to design scalable optimization algorithms with (almost) dimension-free convergence rate guarantees.
We elaborate on this in the next section.

\subsection{The \acl{MP} template}
\label{sec:MP}

We begin by considering the well-known \acdef{MP} method of \citet{Nem04}.
Following \cite{JNT11,MLZF+19}, this is defined via the recursion
\begin{equation}
\label{eq:MP}
\tag{MP}
\begin{aligned}
\lead
	&= \proxof{\curr}{-\curr[\stepalt] \curr[\signal]}
	\\
\next
	&= \proxof{\curr}{-\curr[\stepalt] \lead[\signal]}
\end{aligned}
\end{equation}
where
\begin{enumerate}
\item
$\run=\running$ denotes the method's iteration counter (for the origins of the half-integer notation, see \citet{FP03} and references therein).
\item
$\curr[\stepalt] > 0$ is the algorithm's step-size sequence.
\item
$\curr[\signal]$ and $\lead[\signal]$ are stochastic gradients of $\obj$ obtained by querying the oracle $\orcl$ at $\curr$ and $\lead$ respectively.
\item
$\proxof{\curr}{\cdot}$ is a generalized projection operator known as the method's ``prox-mapping'' (more on this later).
\end{enumerate}


The most elementary instance of \eqref{eq:MP} is the \acdef{EG} algorithm of \citet{Kor76}, in which case the method's prox-mapping is the Euclidean projector
\begin{equation}
\label{eq:Eucl}
\proxof{\point}{\dpoint}
	= \Eucl_{\points}(\point+\dpoint)
	\defeq \argmin\nolimits_{\pointalt\in\points} \twonorm{\point + \dpoint - \pointalt}
\end{equation}
for all $\point\in\points$, $\dpoint\in\dpoints$.
More generally, the prox-mapping in \eqref{eq:MP} is defined in terms of a \emph{Bregman regularizer} as follows:


\begin{table*}[t]
\footnotesize
\renewcommand{\arraystretch}{1.2}
\addtolength{\tabcolsep}{-.4em}
\centering

\begin{tabular}{lcccccc}
\toprule
	&\bfseries Domain ($\points$)
	&\bfseries Breg. Diam. ($\bregdiam$)
	&\bfseries Range ($\range$)
	&\bfseries Shape ($\shape$)
	&\bfseries Rate ($\lips = \infty$)
	&\bfseries Rate ($\lips<\infty$, $\sdev=0$)
	\\
\midrule
\scshape
Euclidean
	&any below
	&$\bigoh(1)$
	&$\bigoh(1)$
	&$\sqrt{\vdim}$
	&$\bigof[\big]{\sqrt{\vdim / \nRuns}}$
	&$\bigof{\vdim / \nRuns}$
	\\
\scshape
Entropic
	&simplex
	&$\infty$
	&$\log\vdim$
	&$1$
	&$\bigof[\big]{\sqrt{\log\vdim / \nRuns}}$
	&$\bigof{\log\vdim / \nRuns}$
	\\
\scshape
Quantum
	&spectrahedron
	&$\infty$
	&$\log\vdim$
	&$1$
	&$\bigof[\big]{\sqrt{\log\vdim / \nRuns}}$
	&$\bigof{\log\vdim / \nRuns}$
	\\
\scshape
ComBand
	&$\conv(\pures)$
	&$\infty$
	&$\bigoh(\log\vdim)$
	&$1$
	&$\bigof[\big]{\sqrt{\log\vdim / \nRuns}}$
	&$\bigof{\log\vdim / \nRuns}$
	\\
\bottomrule
\end{tabular}
\caption{%
The convergence rate of \eqref{eq:MP} in terms of $\vdim$ and $\nRuns$ for different regularizers.
In the combinatorial setup of \cref{ex:comband}, the unnormalized entropy has $\range = \maxsel(1 + \log(\vdim/\maxsel))$, where $\maxsel = \max_{\pure\in\pures} \onenorm{\pure}$ is the maximum number of elements of $\{1,\dotsc,\vdim\}$ that can be simultaneously ``on'' \citep[Chap.~30]{LS20}.
In many applications, $\maxsel$ does not scale with $\vdim$, so it has been absorbed in the $\bigoh(\cdot)$ notation;
other than that, $\bigoh(\cdot)$ contains only universal constants.
}
\label{tab:rates}
\vspace{-1ex}
\end{table*}


\begin{definition}
\label{def:Bregman}
A \emph{Bregman regularizer} on $\points$ is a convex function $\hreg\from\vecspace\to\R\cup\{\infty\}$
such that
\begin{enumerate}
\item
$\dom\hreg = \points$ and $\hreg$ is continuous on $\points$.
\item
The subdifferential of $\hreg$ admits a \emph{continuous selection}, \ie there exists a continuous mapping $\subsel\hreg\from\dom\subd\hreg\to\dpoints$ with $\subsel\hreg(\point) \in \subd\hreg(\point)$ for all $\point\in\dom\subd\hreg$.
\item
$\hreg$ is \emph{strongly convex} on $\points$, \ie
\begin{equation}
\hreg(\pointalt)
	\geq \hreg(\point)
		+ \braket{\subsel\hreg(\point)}{\pointalt - \point}
		+ \tfrac{1}{2} \hstr \norm{\pointalt - \point}^{2}
\end{equation}
for some $\hstr>0$ and all $\point\in\dom\subd\hreg$, $\pointalt\in\points$.
\end{enumerate}
\noindent
We also define the \emph{Bregman divergence} of $\hreg$ as
\begin{equation}
\label{eq:Breg}
\breg(\pointalt,\point)
	= \hreg(\pointalt)
		- \hreg(\point)
		- \braket{\subsel\hreg(\point)}{\pointalt - \point}
\end{equation}
and the induced \emph{prox-mapping} as
\begin{equation}
\label{eq:prox}
\proxof{\point}{\dpoint}
	= \argmin\nolimits_{\pointalt\in\points} \{ \braket{\dpoint}{\point - \pointalt} + \breg(\pointalt,\point) \}
\end{equation}
for all $\point\in\proxdom$, $\pointalt\in\points$ and all $\dpoint\in\dpoints$.
\end{definition}

\begin{remark*}
The set $\proxdom \defeq \dom\subd\hreg$ is often referred to as the \emph{prox-domain} of $\hreg$;
by standard results in convex analysis, we have $\relint\points \subseteq \proxdom \subseteq \points$ \citep[Chap.~26]{Roc70}.
\end{remark*}

In terms of output, the candidate solution returned by \eqref{eq:MP} after $\nRuns$ iterations is the so-called ``ergodic average''
\begin{equation}
\last[\stateavg]
	= \frac
		{\sum_{\run=\start}^{\nRuns} \curr[\stepalt] \lead}
		{\sum_{\run=\start}^{\nRuns} \curr[\stepalt]}.
\end{equation}
Then, assuming the method's step-size $\curr[\stepalt]$ is chosen appropriately (more on this below), $\last[\stateavg]$ enjoys the following guarantees \citep{JNT11,RS13-NIPS}:
\begin{subequations}
\label{eq:rate-MP}
\begin{enumerate}[\itshape a\upshape)]
\item
If $\obj$ satisfies \eqref{eq:BG}, then
\begin{align}
\label{eq:rate-MP-BG}
\exof{\obj(\last[\stateavg]) - \min\obj}
	&= \bigof*{\sqrt{\frac{\gbound^{2} + \sdev^{2}}{\hstr} \frac{\init[\breg]}{\nRuns}}}
\shortintertext{%
\item
If $\obj$ satisfies \eqref{eq:LG}, then}
\label{eq:rate-MP-LG}
\exof{\obj(\last[\stateavg]) - \min\obj}
	&= \bigof*{\frac{\lips \init[\breg]}{\hstr\nRuns} + \sdev\sqrt{\frac{\init[\breg]}{\hstr\nRuns}}}
\end{align}
\end{enumerate}
\end{subequations}

In the above, $\init[\breg] = \breg(\sol,\init)$ is the minimum Bregman divergence between a solution $\sol$ of \eqref{eq:opt} and the initial state $\init$ of \eqref{eq:MP}.
In particular, if \eqref{eq:MP} is initialized at the \emph{prox-center} $\point_{c} = \argmin\hreg$ of $\points$, we have
\begin{equation}
\init[\breg]
	\leq \hreg(\sol) - \min\hreg
	\leq \max\hreg - \min\hreg
	\eqdef \range.
\end{equation}
We will refer to $\range = \max\hreg - \min\hreg$ as the \emph{range} of $\hreg$.

To quantify the interplay betwen the problem's dimensionality and the rate guarantees \eqref{eq:rate-MP} for \eqref{eq:MP}, it will be convenient to introduce the normalized regularity parameters
\begin{equation}
\gbound_{\hreg}
	= \frac{\gbound}{\sqrt{\hstr}}
	\quad
\lips_{\hreg}
	= \frac{\lips}{\hstr}
	\quad
	\text{and}
	\quad
\sdev_{\hreg}
	= \frac{\sdev}{\sqrt{\hstr}}
\end{equation}
and the associated \emph{shape factor}
\begin{equation}
\label{eq:shape}
\shape
	= \begin{cases}
		\sqrt{\gbound_{\hreg}^{2} + \sdev_{\hreg}^{2}}
			&\quad
			\text{if $\lips = \infty$,}
			\\
		\sqrt{\lips_{\hreg}}
			&\quad
			\text{if $\lips<\infty$ and $\sdev=0$,}
			\\
		\sdev_{\hreg}
			&\quad
			\text{if $\lips<\infty$ and $\sdev>0$.}
	\end{cases}		
\end{equation}
Since at least one of the terms $\gbound/\sqrt{\hstr}$, $\lips/\hstr$ and $\sdev/\sqrt{\hstr}$ appears in \eqref{eq:rate-MP}, it follows that the leading term in $\nRuns$ scales as $\bigoh(\shape \sqrt{\init[\breg]/\nRuns})$ in non-smooth\,/\,stochastic environments, and as $\bigoh(\shape^{2} \init[\breg]/\nRuns)$ in smooth, deterministic problems.

The importance of the normalized parameters $\gbound_{\hreg}$, $\lips_{\hreg}$, $\sdev_{\hreg}$ and the shape factor $\shape$ lies in the fact that they do not depend on the ambient norm $\norm{\cdot}$ (a choice which, to a certain extent, is arbitrary).
Indeed, if $\norm{\cdot}$ and $\altnorm{\cdot}$ are two norms on $\vecspace$ that are related as $\norm{\cdot} \leq \mu \altnorm{\cdot}$ for some $\mu>0$, it is straightforward to verify that $\hreg$ is $(\mu^{2}\hstr)$-strongly convex relative to $\altnorm{\cdot}$.
Likewise, in terms of dual norms we have $\dnorm{\cdot} \geq (1/\mu) \altdnorm{\cdot}$, so the constants $\gbound$, $\sdev$ and $\lips$ would respectively become $\mu\gbound$, $\mu\sdev$ and $\mu^{2}\lips$ when computed under $\altnorm{\cdot}$.
In general, these inequalities are all tight, so a change in norm does not affect the shape factor $\shape$;
accordingly, any dependence of $\shape$ on $\vdim$ will be propagated verbatim to the guarantees \eqref{eq:rate-MP}.

In \cref{tab:rates}, we provide the values of $\range$ and $\shape$ for the following cases:
\begin{enumerate}
\item
The \emph{Euclidean regularizer} $\hreg(\point) = \twonorm{\point}^{2}/2$ that gives rise to the \acl{EG} algorithm \eqref{eq:Eucl}.
\item
The \emph{entropic regularizer} $\hreg(\point) = \sum_{\coord=1}^{\nResources} \point_{\coord} \log\point_{\coord}$ for the simplex setup of \cref{ex:simplex}.
\item
The \emph{von Neumann regularizer} $\hreg(\covmat) = \tr\parens{\covmat \log\covmat} + (1 - \tr\covmat) \log(1-\tr\covmat)$ for the spectrahedron setup of \cref{ex:spectron}.
\item
The \emph{unnormalized entropy} $\hreg(\point) = \sum_{\coord=1}^{\nResources} (\point_{\coord} \log\point_{\coord} - \point_{\coord})$ for the combinatorial setup of \cref{ex:comband}.
\end{enumerate}
These derivations are standard, so we omit the details.
For posterity, we only note that the logarithmic dependence on $\vdim$ is asymptotically optimal, \cf \citep{CBL06,BCB12} and references therein.

\subsection{The \unixgrad algorithm}
\label{sec:unixgrad}

As can be seen from \cref{tab:rates}, the \acl{MP} algorithm achieves an almost dimension-free rate of convergence when used with a suitable regularizer.
However, this comes with two important caveats:
First, the algorithm's rate in the smooth case falls short of the optimal $\bigoh(1/\nRuns^{2})$ dependence in $\nRuns$, so \eqref{eq:MP} is suboptimal in this regard.
Second, to achieve the rates presented in \cref{eq:rate-MP}, the algorithm's step-size $\curr[\stepalt]$ must be tuned with prior knowledge of the problem's parameters:
in particular, under \eqref{eq:BG}, the algorithm must be run with step-size $\curr[\stepalt] \propto 1 /\sqrt{(\gbound^{2} + \sdev^{2}) \nRuns}$
while,
under \eqref{eq:LG},
the algorithm requires $\curr[\stepalt] = \hstr/\lips$ if $\sdev=0$ and $\curr[\stepalt] \propto 1/(\sdev\sqrt{\nRuns})$ otherwise.
This creates an undesirable state of affairs because the parameters $\gbound$, $\lips$ and $\sdev$ are usually not known in advance, and \eqref{eq:MP} can \textendash\ and \emph{does} \textendash\ fail to converge if run with an untuned step-size.

In the rest of this section, we briefly discuss the \unixgrad algorithm of \citet{KLBC19} which expands on the \acl{MP} template in the following two crucial ways:
\begin{enumerate*}
[\itshape a\upshape)]
\item
it introduces an iterate-averaging mechanism in the spirit of \citet{Cut19} to enable acceleration;
and
\item
it employs an adaptive step-size policy that does not require any tuning by the optimizer.
\end{enumerate*}
In so doing, \unixgrad interpolates smoothly between the optimal convergence rates described in \cref{sec:prelims} without requiring any prior knowledge of $\gbound$, $\lips$ or $\sdev$.

Concretely, \unixgrad proceeds as \eqref{eq:MP}, but instead of querying $\orcl$ at $\curr$ and $\lead$, it introduces the weighted query states
\begin{equation}
\label{eq:averages}
\begin{aligned}
\curr[\stateavg]
	&= \frac
		{\curr[\step]\curr + \sum_{\runalt=\start}^{\run-1} \iter[\step] \iterlead}
		{\sum_{\runalt=\start}^{\run} \iter[\step]}
	\\
\lead[\stateavg]
	&= \frac
		{\curr[\step]\lead + \sum_{\runalt=\start}^{\run-1} \iter[\step] \iterlead}
		{\sum_{\runalt=\start}^{\run} \iter[\step]}
\end{aligned}
\end{equation}
where $\curr[\step]$ is a ``gradient weighting'' parameter.
Then, building on an idea by \citet{RS13-COLT,RS13-NIPS}, the oracle queries $\curr[\signal] \gets \orcl(\curr[\stateavg];\curr[\sample])$ and $\lead[\signal] \gets \orcl(\lead[\stateavg];\lead[\sample])$ are used to update the method's step-size as
\begin{equation}
\label{eq:step-unix}
\curr[\stepalt]
	= \frac
		{\bregdiam \curr[\step]}
		{\sqrt{1 + \sum_{\runalt=1}^{\run-1} \iter[\step]^{2} \dnorm{\iterlead[\signal] - \iter[\signal]}^{2}}}
\end{equation}
where
\begin{equation}
\label{eq:bregdiam}
\bregdiam
	= \sup\nolimits_{\point\in\points,\pointalt\in\proxdom}
		\sqrt{2 \bregof{\point}{\pointalt}}
\end{equation}
is the so-called \emph{Bregman diameter} of $\points$.

With all this in hand, \citet{KLBC19} provide the following bounds if \unixgrad is run with $\curr[\step] = \run$:
\begin{subequations}
\label{eq:rate-unix}
\begin{enumerate}[\itshape a\upshape)]
\item
If $\obj$ satisfies \eqref{eq:BG}, then
\begin{align}
\label{eq:rate-unix-BG}
\hspace{-1.7em}
\exof{\obj(\lastlead[\stateavg]) - \min\obj}
	&= \bigof[\bigg]{\frac{\bregdiam\sqrt{\gbound^{2} + \sdev^{2}}}{\sqrt{\hstr\nRuns}}}
\shortintertext{%
\item
If $\obj$ satisfies \eqref{eq:LG}, then}
\label{eq:rate-unix-LG}
\hspace{-1.7em}
\exof{\obj(\lastlead[\stateavg]) - \min\obj}
	&= \bigof*{\frac{\bregdiam^{2}\lips}{\hstr\nRuns^{2}} + \frac{\bregdiam \sdev}{\sqrt{\hstr\nRuns}}}
\end{align}
\end{enumerate}
\end{subequations}
As we mentioned in \cref{sec:prelims}, the bounds \eqref{eq:rate-unix} cannot be improved in terms of $\nRuns$ without further assumptions, so \unixgrad is universally optimal in this regard.

That being said, these guarantees also uncover an important limitation of \unixgrad,
namely
that the bounds \eqref{eq:rate-unix} become void when the method is used in conjunction with one of the non-Euclidean frameworks of \cref{ex:simplex,ex:spectron,ex:comband}.
For example, the Bregman diameter of the simplex under the entropic regularizer is $\bregdiam = \sup_{\point,\pointalt} \sum_{\coord} \point_{\coord} \log(\point_{\coord}/\pointalt_{\coord}) = \infty$, so the multiplicative constants in \eqref{eq:rate-unix} become infinite (and the bounds themselves become meaningless).
However, since the use of these regularizers is crucial to obtain the scalable, dimension-free convergence rates reported in \cref{tab:rates},
\footnote{In particular, since the shape factor of the Euclidean regularizer is $\shape = \sqrt{\vdim}$, employing \unixgrad with ordinary Euclidean projections would not lead to scalable guarantees.}
we are led to the open question we stated before:
\begin{center}
\itshape
Is it possible to achieve almost dimension-free convergence rates\\
while retaining an order-optimal dependence on $\nRuns$?
\end{center}
\vspace{-1ex}
We address this question in the next section.

\section{Universal dual extrapolation}
\label{sec:undergrad}

The point of departure of our analysis is the observation that gradient queries enter \eqref{eq:MP} with \emph{decreasing} weights.
Specifically, if \unixgrad is run with $\curr[\step] = \run$ (a choice which is necessary to have a shot at acceleration), the denominator of \eqref{eq:step-unix} may grow as fast as $\Theta(\run^{3/2})$ in the non-smooth/stochastic case, leading to an asymptotic $\bigoh(1/\sqrt{\run})$ worst-case behavior for $\curr[\stepalt]$.
In fact, even under the ansatz that the algorithm's query points converge to a minimizer of $\obj$ at an accelerated rate, the denominator of \eqref{eq:step-unix} may still grow as $\Theta(\run)$,
indicating that $\curr[\stepalt]$ will, at best, stabilize to a positive value as $\run\to\infty$.
This feature of the step-size rule \eqref{eq:step-unix} is somewhat counterintuitive because conventional wisdom would suggest that
\begin{enumerate*}
[\itshape a\upshape)]
\item
recent queries are more useful than older, potentially obsolete ones;
and
\item
gradients should be ``inflated'' as the method's query points approach a zero-gradient solution in order to maintain a fast rate of convergence.
\end{enumerate*}

The problem with a vanishing step-size becomes especially pronounced if the method is used with a non-Euclidean regularizer (which is what one would wish to do in order to obtain scalable convergence guarantees).
To see this, consider the iterates of the \acl{MP} template generated by the regularizer $\hreg(\point) = \point\log\point$ on $\points = [0,\infty)$.%
\footnote{Strictly speaking this regularizer is not strongly convex over $[0,\infty)$ but this detail is not relevant for the question at hand.}
In this case, the induced prox-mapping is $\proxof{\point}{\dpoint} = \point \exp(\dpoint)$, leading to the recursion
\begin{equation}
\new\point
	= \proxof{\point}{-\stepalt\dvec}
	= \point \exp(-\stepalt\dvec).
\end{equation}
Therefore, if the problem's objective function attains its minimum at $0$, the actual steps of the method scale as $\new\point-\point = \bigoh(\point)$ for small $\point$, so it is imperative to maintain a large step-size to avoid stalling the algorithm.

This scaling issue is at the heart of the \acdef{DE} method of \citet{Nes07}.
Originally designed to solve \aclp{VI} and related problems, the method proceeds by
\begin{enumerate*}
[(\itshape i\upshape\hspace*{1pt})]
\item
using a prox-step to generate the method's leading state and get a ``look-ahead'' gradient query;
\item
aggregating gradient information with a \emph{constant} weight;
and, finally,
\item
using a ``primal-dual'' mirror map to update the method's base state.
\end{enumerate*}
Formally, the algorithm follows the iterative update rule
\begin{equation}
\label{eq:DE}
\tag{DE}
\begin{aligned}
\lead
	&= \proxof{\curr}{-\curr[\stepalt] \curr[\signal]}
	\\
\next[\dstate]
	&= \curr[\dstate] - \lead[\signal]
	\\
\next
	&= \mirror(\next[\stepalt]\next[\dstate])
\end{aligned}
\end{equation}
where the so-called ``mirror map'' $\mirror\from\dpoints\to\points$ is defined as
\begin{equation}
\label{eq:mirror}
\mirror(\dpoint)
	= \argmax_{\point\in\points} \{ \braket{\dpoint}{\point} - \hreg(\point) \}.
\end{equation}

Unfortunately, the template \eqref{eq:DE} is not sufficient for our purposes, for two main reasons:
First, the method still couples a prox-step with a variable (decreasing) step-size update;
this is not problematic for the application of the method to \acsp{VI} (where the achievable rates are different), but it is not otherwise favorable for acceleration.

In addition to the above,
the method's gradient pre-multiplier is the same as its post-multiplier ($\curr[\stepalt]$ in both cases), and it is not possible to differentiate these parameters while maintaining optimal rates \citep{Nes07}.
However, this differentiation is essential for acceleration, especially when $\curr[\stepalt]$ cannot be tuned with prior knowledge of the problem's parameters.

Our approach to overcome this issue consists of:
\begin{enumerate*}
[\itshape a\upshape)]
\item
eliminating the prox-step altogether in favor of a mirror step;
and
\item
separating the weights used for introducing new gradients to the algorithm versus those used to generate the base and leading states.
\end{enumerate*}
To formalize this, we introduce below the ``universal'' \acl{DE} template:
\begin{equation}
\label{eq:UDE}
\tag{UDE}
\begin{alignedat}{2}
\lead[\dstate]
	&= \curr[\dstate] - \curr[\step] \curr[\signal]
	&\quad
\lead
	&= \mirror(\curr[\learn]\lead[\dstate])
	\\
\next[\dstate]
	&= \curr[\dstate] - \curr[\step] \lead[\signal]
	&\quad
\next
	&= \mirror(\next[\learn] \next[\dstate])
\end{alignedat}
\end{equation}
In the above, the gradient signals $\curr[\signal]$ and $\lead[\signal]$ are considered generic and the query points are not specified.
To get a concrete algorithm, we will use the weighting scheme of \citet{KLBC19} and query the oracle at the averaged states $\curr[\avg\state]$ and $\lead[\avg\state]$ introduced previously in \eqref{eq:averages}.
Finally, regarding the algorithm's gradient weighting and averaging parameters ($\curr[\step]$ and $\curr[\learn]$ respectively), we will use an \emph{increasing} weight for the method's step-size $\curr[\step] = \run$ and the adaptive rule
\begin{equation}
\label{eq:learn}
\curr[\learn]
	= \frac{\numer}{\sqrt{\denom + \sum_{\runalt=\start}^{\run-1} \iter[\step]^{2} \dnorm{\iterlead[\signal] - \iter[\signal]}^{2}}}
\end{equation}
for the method's learning rate (the parameters $\denomsqrt$ and $\numer$ are discussed below, and we are using the standard convention that empty sums are taken equal to zero).


\begin{algorithm}[tbp]
\DontPrintSemicolon
\small
\addtolength{\baselineskip}{2pt}

\textbf{Parameters}
	$\denomsqrt \gets \sqrt{\hstr}$;
	$\numer \gets \sqrt{\hstr(\range + \hstr \diampoints^{2})}$\;
\textbf{Initialize}
	$\init[\dstate] \gets 0$;
	$\init[\anchor] \gets 0$;
	$\init[S] \gets \denom$\;
	%
\For{$\run = \running,\nRuns$}{%
	$ \curr[\learn] \gets \numer /\sqrt{\curr[S]}$
	\tcp*{set learning rate}
	$\curr[\state] \gets \mirror \parens*{\curr[\learn] \curr[\dstate] }$
	\tcp*{mirror step}
	$\curr[\avg{\state}] \gets \parens*{\curr[\step] \curr[\state] + \curr[\anchor] } \big/ { \sum_{\runalt = 1}^{\run} \iter[\step] }$
	\tcp*{mixing}
	$\curr[\signal] \gets \orcl(\curr[\avg{\state}];\curr[\sample])$ 
	\tcp*{oracle query}
	$\lead[\dstate] \gets \curr[\dstate] - \curr[\step] \curr[\signal]$
	\tcp*{dual step}
	$\lead \gets \mirror \parens*{ \curr[\learn] \lead[\dstate]} $
	\tcp*{mirror step}
	$\lead[\avg\state] \gets \parens*{\curr[\step] \lead + \curr[\anchor] } \big/ { \sum_{\runalt = 1}^{\run} \iter[\step] }$
	\tcp*{mixing}
	$\lead[\signal] \gets \orcl(\lead[\avg\state];\lead[\sample])$
	\tcp*{oracle query}
	$\next[\dstate] \gets \curr[\dstate] - \curr[\step] \lead[\signal]$
	\tcp*{dual step}
	$\next[S] \gets \curr[S] + \curr[\step]^{2} \dnorm{\lead[\signal] - \curr[\signal]}^{2} $
	\tcp*{precondition}
	$\next[\anchor] \gets \curr[\anchor] + \curr[\step] \lead$
	\tcp*{update mixing state}
}

\Return{$\last[\out] \gets \lastlead[\avg\state]$}

\caption{\Acf{undergrad}}
\label{alg:undergrad}
\end{algorithm}


The resulting method \textendash\ which we call \acdef{undergrad} \textendash\ is encoded in pseudocode form in \cref{alg:undergrad} and represented schematically in \cref{fig:undergrad}.
Its main guarantees are as follows:

\begin{theorem}
\label{thm:undergrad}
Suppose that \ac{undergrad} \textpar{\cref{alg:undergrad}} is run for $\nRuns$ iterations with $\curr[\learn]$ given by \eqref{eq:learn},
$\curr[\step] = \run$ for all $\run=\running$,
and
$\denomsqrt = \sqrt{\hstr}$,
$\numer = \hconst \sqrt{\hstr}$ with $\hconst = \sqrt{\range + \hstr \diampoints^{2}}$.
Then the algorithm's output state $\last[\out] \equiv \lastlead[\avg\state]$ simultaneously enjoys the following guarantees:
\begin{subequations}
\label{eq:rate-under}
\begin{enumerate}[\itshape a\upshape)]
\item
If $\obj$ satisfies \eqref{eq:BG}, then
\begin{equation}
\label{eq:rate-under-BG}
\!\hspace{-1em}
\exof{\obj(\last[\out])}
	\leq \min\obj
		+ 2\hconst \sqrt{\frac{\hstr + 8(\gbound^{2}+\sdev^{2})}{\hstr\nRuns}}
\end{equation}
\item
If $\obj$ satisfies \eqref{eq:LG}, then
\begin{equation}
\label{eq:rate-under-LG}
\exof{\obj(\last[\out])}
	\leq \min\obj
		+ \frac{32\sqrt{2}\hconst^{2}\lips}{\hstr\nRuns^{2}}
		+ \frac{8\sqrt{2}\hconst \sdev}{\sqrt{\hstr\nRuns}}
\end{equation}
\end{enumerate}
\end{subequations}
\end{theorem}

\cref{thm:undergrad} is our main result so, before discussing its proof (which we carry out in detail in the appendix), some remarks are in order.


\begin{figure}[tbp]
\centering
\includegraphics[width=.5\columnwidth]{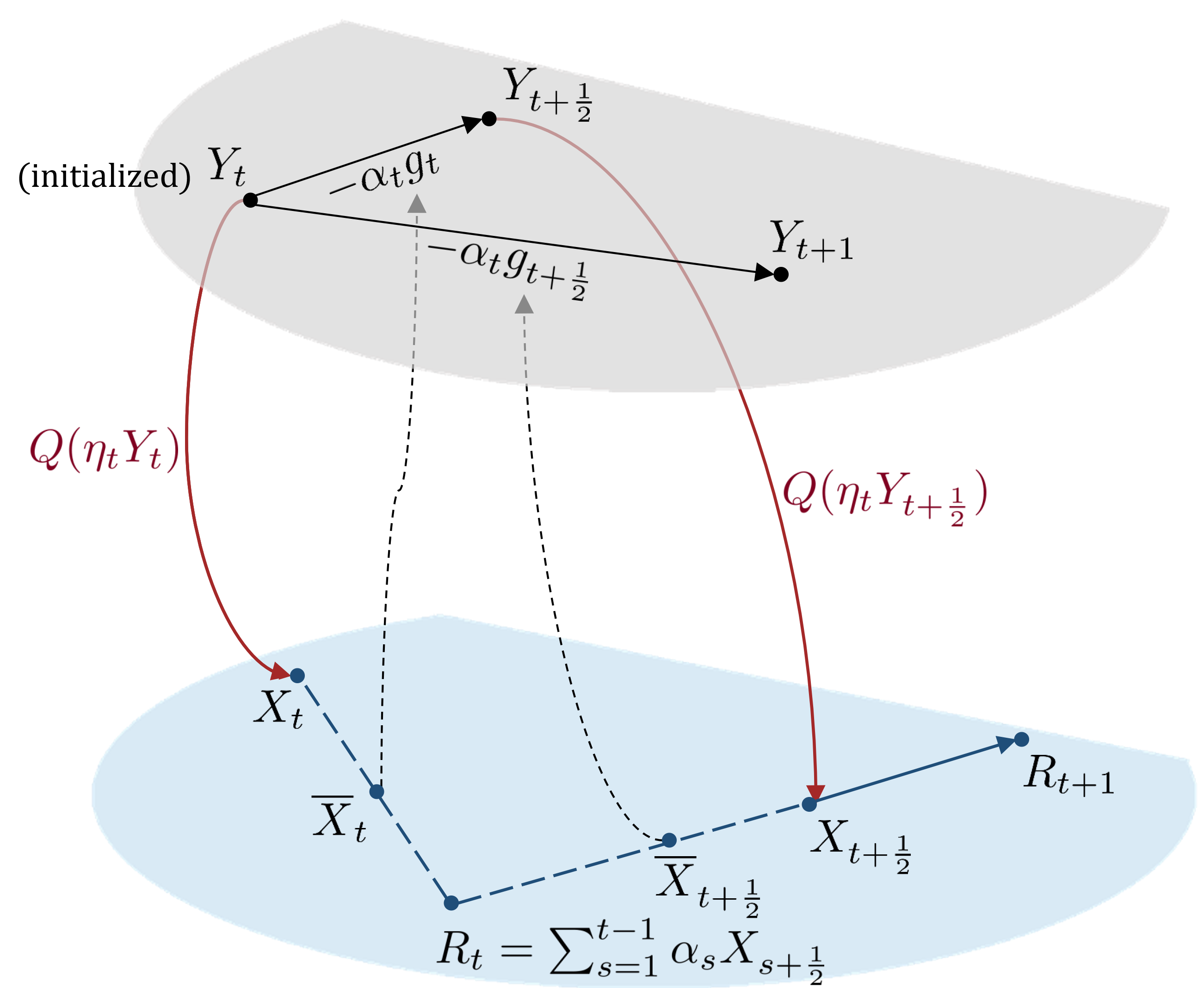}
\caption{Schematic representation of the \ac{undergrad} algorithm (\cref{alg:undergrad}).
The light blue area represents the problem's domain ($\points$), whereas the grey area represents the dual space ($\dpoints$).}
\label{fig:undergrad}
\vspace{-1ex}
\end{figure}


The first point of note concerns the dependence of the anytime bounds \eqref{eq:rate-under} on the problem's dimensionality.
To that end, let $\underconst_{\fast} = \hconst^{2}$ and $\underconst_{\slow} = \hconst$, so \ac{undergrad}'s rate of convergence scales as
$\bigoh(\underconst_{\fast} \shape^{2}/\nRuns^{2})$ in smooth, deterministic problems,
and as
$\bigoh(\underconst_{\slow} \shape/\sqrt{\nRuns})$ in non-smooth and/or stochastic environments.
Thus, to compare the algorithm's rate of convergence to that of \acl{MP} and \unixgrad (and up to universal constants), we have to compare $\hconst$ to $\range$ and $\bregdiam$ respectively.

To that end, we calculate below the values of $\underconst_{\fast}$ and $\underconst_{\slow}$ in the three archetypal examples of \cref{sec:base}:
\begin{enumerate}
\item
In the simplex setup of \cref{ex:simplex}, we have $\range = \log\vdim$, $\diampoints = 1$ and $\hstr = 1$, so $\underconst_{\slow} = \bigoh(\sqrt{\log\vdim})$ and $\underconst_{\fast} = \bigoh(\log\vdim)$.
\item
In the spectrahedron setup of \cref{ex:spectron}, we have again $\range = \log\vdim$, $\diampoints = 1$ and $\hstr = 1$, so $\underconst_{\slow} = \bigoh(\sqrt{\log\vdim})$ and $\underconst_{\fast} = \bigoh(\log\vdim)$.
[For a detailed discussion, see \cite{MBNS17,BMB20,KSST12} and references therein.]
\item
Finally, in the combinatorial setup of \cref{ex:comband}, we have $\range = \maxsel(1 + \log(\vdim/\maxsel))$, $\diampoints = \maxsel$ and $\hstr=1$ \cite{LS20}.
Thus, if $\maxsel = \bigoh(1)$ in $\vdim$, we get again $\underconst_{\slow} = \bigoh(\sqrt{\log\vdim})$ and $\underconst_{\fast} = \bigoh(\log\vdim)$.
\end{enumerate}

The above shows that \ac{undergrad} achieves the desired almost dimension-free rates of the \emph{non-adaptive} \acl{MP} algorithm, as well as the \emph{universal} order-optimal guarantees of \unixgrad.
The only discrepancy with the rates presented in \cref{tab:rates} is the additive constant $\hstr$ that appears in the numerator of \eqref{eq:rate-under-BG}:
this constant is an artifact of the analysis and it only becomes relevant when $\gbound\to0$ and $\sdev\to0$.
Since the scaling of the algorithm's convergence rate concerns the large $\gbound$ regime, this difference is not relevant for our purposes.

An additional difference between \ac{undergrad} and \unixgrad is that the latter involves the prox-mapping \eqref{eq:prox}, whereas the former involves the mirror map \eqref{eq:mirror}.
To compare the two in terms of their per-iteration complexity, note that if we apply the prox-mapping \eqref{eq:prox} to the prox-center $\point_{c} \gets \argmin\hreg$ of $\points$, we get
\begin{align}
\proxof{\point_{c}}{\dpoint}
    &= \argmin\nolimits_{\point\in\points}
        \{ \braket{\dpoint}{\point_{c} - \point} + \breg(\point,\point_{c}) \}
    \notag\\
    &= \argmin\nolimits_{\point\in\points}
        \{ \hreg(\point) - \braket{\nabla\hreg(\point_{c}) + \dpoint}{\point} \}
    \notag\\
    &= \mirror(\nabla\hreg(\point_{c}) + \dpoint)
\end{align}
so, in particular, $\mirror(\dpoint) = \proxof{\point_{c}}{\dpoint}$ whenever $\point_{c}\in\relint\points$ (which is the case for most regularizers used in practice, including the Legendre regularizers used in \crefrange{ex:simplex}{ex:comband} above).
This shows that the calculation of the mirror map $\mirror(\dpoint) = \proxof{\point_{c}}{\dpoint - \nabla\hreg(\point_{c})}$ is at least as simple as the calculation of the prox-mapping $\proxof{\point}{\dpoint}$ for a general base point $\point\in\points$ \textendash\ and, typically, calculating the mirror map is strictly lighter because there is no need to vary the base point over different iterations of the algorithm.
In this regard, the per-iteration overhead of \eqref{eq:UDE} is actually \emph{lighter} compared to \eqref{eq:MP} or \eqref{eq:DE}.

Finally, we should note that all our results above implicitly assume that the problem's domain is bounded (otherwise the range parameter $\range$ of the problem becomes infinite).
Thus, in addition to these convergence properties of \ac{undergrad}, we also provide below an asymptotic guarantee for problems with an \emph{unbounded} domain:
\begin{theorem}
\label{thm:unbounded}
Suppose that \ac{undergrad} is run with perfect oracle feedback with $\curr[\learn]$ given by \eqref{eq:learn} and $\curr[\step] = \run$.
If $\obj$ satisfies \eqref{eq:LG}, the algorithm's output state $\last[\out] = \lastlead[\avg\state]$ enjoys the rate $\obj(\last[\out]) - \min\obj = \bigoh(1/\nRuns^{2})$.
\end{theorem}

This result provides an important extension of \cref{thm:undergrad} to problems with unbounded domains.
It remains an open question for the future to derive the precise constants in the convergence rate presented in \cref{thm:unbounded}.

\para{Main ideas of the proof}

The detailed proof of \cref{thm:undergrad} is fairly long so we defer it to the appendix and only present here the main ideas.

The main ingredient of our proof is a specific \emph{template inequality} used to derive an ``appropriate'' upper bound of the term $\regstoch_{\nRuns}(\point) \defeq  \sum_{\run=1}^{\nRuns} {\curr[\step]}\inner*{\lead[\signal] }{\lead-\point}$. Importantly, to prove the dimension-free properties of \ac{undergrad}, such an upper-bound \emph{cannot involve Bregman divergences}:
even though this is common practice in previous papers \cite{KLBC19,AllenOrecchia2017}, these terms would ultimately lead to the Bregman diameter $\bregdiam$ that we seek to avoid.
This is a principal part of the reason for switching gears to the \ac{DE} template for \ac{undergrad}:
in so doing, we are able to employ the notion of the \emph{Fenchel coupling} \cite{MS16,MerSta18}, which is a ``primal-dual distance'' as opposed to the Bregman divergence which is a ``primal-primal distance'' (\cf \cref{app:Bregman}).
This poses another challenge on connecting the Fenchel coupling of targeted points before and after a mirror step, for which we need to employ a primal-dual version of the ``three-point identity'' (\cref{lem:3points}).
These elements lead to the following proposition:

\begin{restatable}{proposition}{template}
\label{prop:template}
For all $\point \in \points$, we have
\begin{align}
\regstoch_{\nRuns}(\point)
	&\leq \frac{\range}{\learn_{\nRuns+1}}
	+  \sum_{\run=1}^{\nRuns}\curr[\step]\braket{\lead[\signal] -   \curr[\signal]  }{\lead -  \next[\state] }
    \notag\\
	&- \hstr \sum_{\run=\start}^{\nRuns}\frac{\norm{\next[\state] - \lead}^{2} + \norm{\lead - \curr[\state]}^{2}}{{2\curr[\learn]}}
\label{eq:Propo_energy_inequ}
\end{align}
\end{restatable} 


With \eqref{eq:Propo_energy_inequ} in hand, \eqref{eq:rate-under-BG} comes from the application of the Fenchel-Young inequality to upper-bound the right-hand-side of \eqref{eq:Propo_energy_inequ} as $\sum_{\run=\start}^{\nRuns}\curr[\step]^2 \next[\learn] \dnorm{\lead[\signal] - \curr[\signal]}$ (plus a constant term involving $\diampoints$). The challenge here is to notice and successfully prove that this summation is actually upper-bounded by ${\learn}_{\nRuns+1}^{-1}$ (due to our special choice of the learning rate update). Finally, by its definition, ${\learn}_{\nRuns+1}^{-1}$ can be bounded by $\gbound$, $\sdev$ and $\hstr$ as described in the statement of \cref{thm:undergrad}.

The proof of \eqref{eq:rate-under-LG} is far more complex.
The main challenge is to manipulate the terms in \eqref{eq:Propo_energy_inequ} to derive an upper-bound of the form $\sum_{\run=1}^{\nRuns}\curr[\step]^2 \func(\next[\learn]) \dnorm{\nabla \obj (\lead[\avg\state]) - \nabla \obj \curr[\avg\state]}^2$ (plus a term involving the noise level $\sdev$) where $\func(\next[\learn])$ is a function of the learning rate chosen such that only the first $\nRuns_0 \ll \nRuns$ elements of this summation are positive.
Once this has been achieved, the quantity $\dnorm{\nabla \obj (\lead[\avg\state]) - \nabla \obj \curr[\avg\state]}$ is connected to $\diampoints$ via \eqref{eq:LG} and our claim is obtained.


\section{Numerical Experiments}
\label{sec:numerics}

For the experimental validation of our results, we focus on the simplex setup of \cref{ex:simplex} with linear losses and $\vdim=100$.
Our first experiment concerns the \emph{perfect} \ac{SFO} case and tracks down the convergence properties of \undergrad run with the \emph{entropic regularizer} adapted to the simplex.
As a baseline, we ran \unixgrad, also with the entropic regularizer.
A first challenge here is that the Bregman diameter $\bregdiam$ of the simmplex is infinite, so \unixgrad is not well-defined.
On that account, we choose the step-size update rule of \unixgrad such that its initial step-size $\init[\stepalt]$ equals the initial learning rate $\init[\learn]$ of \undergrad. 
We also ran \unixgrad with the update rule such that $\init[\stepalt]$ is smaller or larger than $\init[\learn]$.
Finally, for comparison purposes, we also present a non-adaptive \emph{accelerated entropic gradient} (AEG) algorithm, and we report the results in \cref{fig:static}.
\begin{figure}[tbp]
	\centering
		\begin{tikzpicture}
			\node (img){\includegraphics[width = .5\columnwidth]{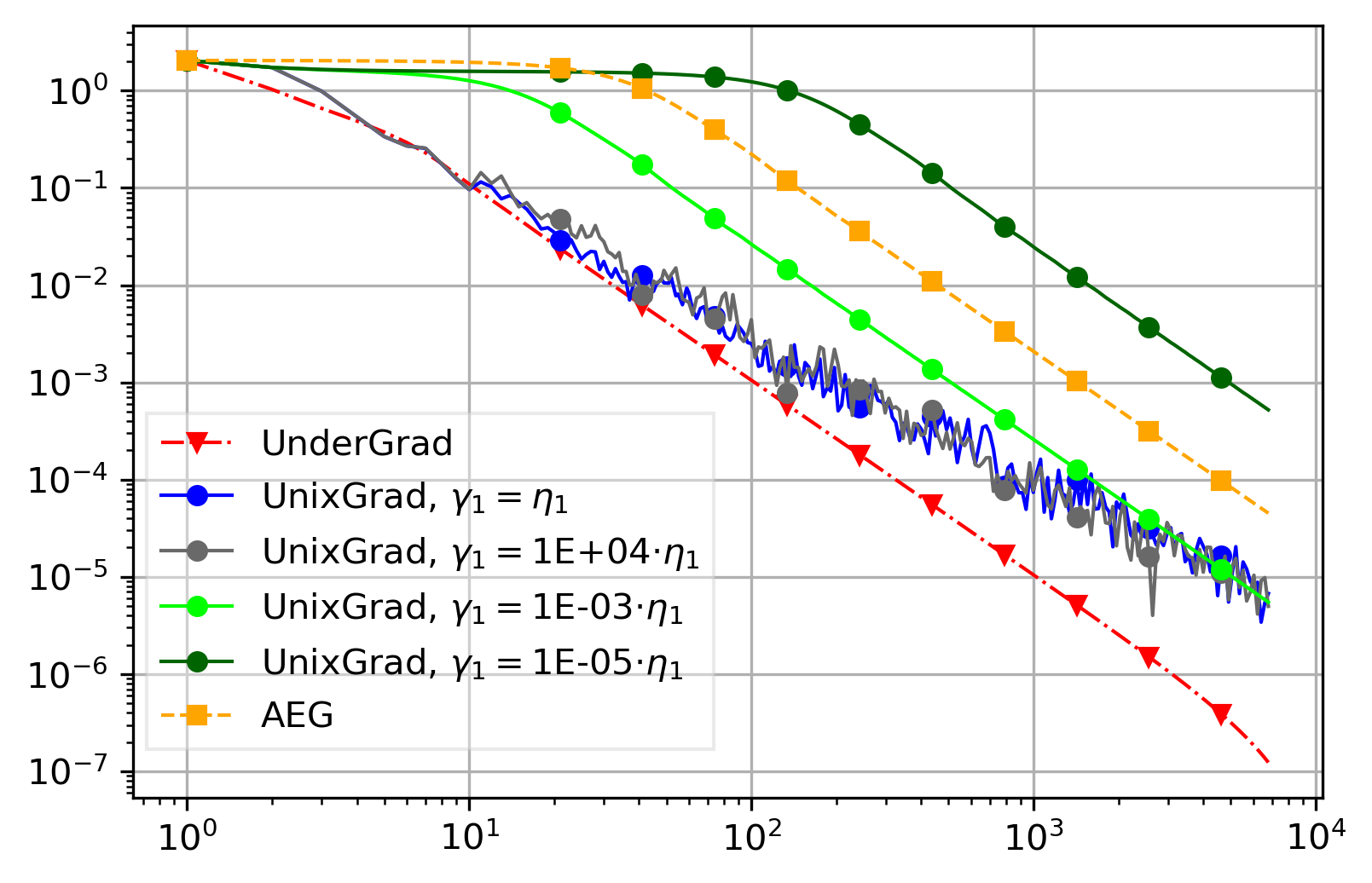}};
			\node[below=of img, node distance=0cm, yshift=1.2cm, xshift =0.2cm] {\scriptsize${\nRuns}$};
			\node[left=of img, node distance=0cm, rotate=90, anchor=center,xshift=0cm, yshift=-1.1cm] {\scriptsize${\gap(\nRuns)}$};
		\end{tikzpicture}
		\vspace{-3ex}
		\caption{\footnotesize Convergence of \undergrad and \unixgrad in the simplex setup with a perfect \ac{SFO}. The $y$-axis corresponds to the differences between the $\obj$-value of the relevant point of each algorithm and $\min \obj$.
		The code is available at \url{https://github.com/dongquan-vu/UnDerGrad_Universal_CnvOpt_ICML2022}.}
		\label{fig:static}
		\vspace{-3ex}
\end{figure}

\cref{fig:static} confirms that \undergrad successfully converges with an accelerated rate of $\bigoh(1/\nRuns^2)$.
Perhaps surprisingly, it also shows that when \unixgrad's initial step-size is small (10E-3 or smaller), \unixgrad also achieves an $\bigoh(1/\nRuns^{2})$, but at a much more conservative pace, trailing \ac{undergrad} by one or two orders of magnitude.
On the other hand, when \unixgrad's initial step-size is of the same magnitude as the \undergrad's learning rate (or larger), \unixgrad eventually destabilizes and its rate drops from $\bigoh(1/\nRuns^2)$ to approximately $\bigoh(1/\nRuns)$.
We also conducted experiments in the setup with a noisy \ac{SFO};
these are reported in \cref{app:numerics}.

\appendix
\setcounter{remark}{0}
\numberwithin{equation}{section}		
\numberwithin{lemma}{section}		
\numberwithin{proposition}{section}		
\numberwithin{theorem}{section}		
\numberwithin{corollary}{section}		

\section{Bregman regularizers and preliminary results}

\subsection{Bregman regularizers and their properties}
\label{app:Bregman}

We begin by clarifying and recalling some of the notational convetions used throughout the paper.
We also give the formal definition of the Fenchel coupling (a key notion for the proof our main results) and we present some preliminary results to prepare the ground for the sequel.

The convex conjugate $\hreg^{\ast}\from\dpoints\to\R$ of $\hreg$ is then defined as
\begin{equation}
\label{eq:conj}
\hreg^{\ast}(\dpoint)
	= \sup_{\point\in\points} \{ \braket{\dpoint}{\point} - \hreg(\point) \}.
\end{equation}
  As a result, the supremum in \eqref{eq:conj} is always attained, and $\hreg^{\ast}(\dpoint)$ is finite for all $\dpoint\in\dpoints$ \citep{BC17}.
Moreover, by standard results in convex analysis \citep[Chap.~26]{Roc70}, $\hreg^{\ast}$ is differentiable on $\dpoints$ and its gradient satisfies the identity
\begin{equation}
\label{eq:dconj}
\nabla\hreg^{\ast}(\dpoint)
	= \argmax_{\point\in\points} \{ \braket{\dpoint}{\point} - \hreg(\point) \}.
\end{equation}
Thus, recalling the definition of the mirror map $\mirror\from\dpoints\to\points$, we readily get
\begin{equation}
\label{eq:mirror2}
\mirror(\dpoint)
	= \nabla\hreg^{\ast}(\dpoint).
\end{equation}
\begin{lemma}
\label{lem:mirror}
Let $\hreg$ be a Bregman regularizer on $\points$.
Then, for all $\point,\new\in\dom\subd\hreg$ and all $\dpoint,\dvec\in\dpoints$, we have:
\begin{subequations}
\label{eq:links}
\begin{alignat}{4}
\label{eq:links-mirror}
&a)&
	&\;\;
	\point = \mirror(\dpoint)
	&\;\iff\;
	&\dpoint \in \subd\hreg(\point).
	\\
\label{eq:links-prox}
&b)&
	&\;\;
	\new\point = \mirror(\nabla\hreg(\point) + \dvec)
	&\;\iff\;
	&\nabla\hreg(\point) + \dvec \in \subd\hreg(\new\point)
\end{alignat}
\end{subequations}
Finally, if $\point = \mirror(\dpoint)$ and $\base\in\points$, we have
\begin{equation}
\label{eq:selection}
\braket{\nabla\hreg(\point)}{\point - \base}
	\leq \braket{\dpoint}{\point - \base}.
\end{equation}
\end{lemma}
\begin{proof}[Proof of \cref{lem:mirror}]
To prove \eqref{eq:links-mirror}, note that $\point$ solves \eqref{eq:dconj} if and only if $\dpoint - \subd\hreg(\point) \ni 0$, \ie if and only if $\dpoint\in\subd\hreg(\point)$.
\cref{eq:links-prox} is then obtained in the same manner.

For the inequality \eqref{eq:selection}, it suffices to show it holds for all $\base\in\proxdom \equiv \dom\subd\hreg$ (by continuity).
To do so, let
\begin{equation}
\phi(t)
	= \hreg(\point + t(\base-\point))
	- \bracks{\hreg(\point) +  \braket{\dpoint}{\point + t(\base-\point)}}.
\end{equation}
Since $\hreg$ is strongly convex relative to $\gmat$ and $\dpoint\in\subd\hreg(\point)$ by \eqref{eq:links-mirror}, it follows that $\phi(t)\geq0$ with equality if and only if $t=0$.
Moreover, note that $\psi(t) = \braket{\nabla\hreg(\point + t(\base-\point)) - \dpoint}{\base - \point}$ is a continuous selection of subgradients of $\phi$.
Given that $\phi$ and $\psi$ are both continuous on $[0,1]$, it follows that $\phi$ is continuously differentiable and $\phi' = \psi$ on $[0,1]$.
Thus, with $\phi$ convex and $\phi(t) \geq 0 = \phi(0)$ for all $t\in[0,1]$, we conclude that $\phi'(0) = \braket{\nabla\hreg(\point) - \dpoint}{\base - \point} \geq 0$.
\end{proof}

As we mentioned earlier, much of our analysis revolves around a ''primal-dual'' divergence between a target point $\base \in \points$ and a dual vector $\dpoint \in \dpoints$, called the \emph{Fenchel coupling}.
Following \citep{MZ19}, this is defined as follows for all $\base\in\points$, $\dpoint\in\dpoints$:
\begin{equation}
\fench(\base,\dpoint)=\hreg(\base)+\hreg^{\ast}(\dpoint)-\braket{\dpoint}{\base}.
\end{equation}
The following lemma illustrates some basic properties of the Fenchel coupling:

\begin{lemma}
\label{lem:Fenchel}
Let $\hreg$ be a Bregman regularizer on $\points$ with convexity modulus $\hstr$. Then, for all $\base \in 
\points$ and all $\dpoint \in \dpoints$, we have:
\begin{enumerate}
\item
$\fench(\base,\dpoint) = \breg(\base,\mirror(\dpoint))$
	if $\mirror(\dpoint) \in \proxdom$ (but not necessarily otherwise).
\item
$\fench(\base,\dpoint) \geq \frac{\hstr}{2} \norm{\mirror(\dpoint)-\base}^{2}$
\end{enumerate} 
\end{lemma}

\begin{proof}
For our first claim, let $\point=\mirror(\dpoint)$. Then, by definition we have:
\begin{equation}
\fench(\base,\dpoint)
	= \hreg(\base)
		- \braket{\dpoint}{\mirror(\dpoint)}
		- \hreg(\mirror(\dpoint))
		- \braket{\dpoint}{\base}
	= \hreg(\base)
		- \hreg(\point)
		- \braket{\dpoint}{\base-\point}.
\end{equation}
Since $\dpoint \in \partial \hreg(\point)$, we have $\hreg'(\point;\base-\point)=\braket{\dpoint}{\base-
\point}$ whenever $\point \in \proxdom$, thus proving our first claim.
For our second claim, working in the previous spirit we get that:
\begin{equation}
\fench(\base,\dpoint)
	= \hreg(\base) - \hreg(\point)
	- \braket{\dpoint}{\base - \point}
\end{equation}
Thus, we obtain the result by recalling the strong convexity assumption for $\hreg$ with respect to the respetive norm $\norm{\cdot}$.
\end{proof}
We continue with some basic relations connecting the Fenchel coupling relative to a target point before and after a gradient step.
The basic ingredient for this is a primal-dual analogue of the so-called ``three-point identity'' for Bregman functions \citep{CT93}:

\begin{lemma}
\label{lem:3points}
Let $\hreg$ be a Bregman regularizer on $\points$. Fix some $\base \in\points$ and let $\dpoint,\dpoint^{+} \in \dpoints$.
Then, letting $\point = \mirror(\dpoint)$, we have
\begin{equation}
\label{eq:3points}
\fench(\base,\dpoint^{+})
	= \fench(\base,\dpoint)
	+ \fench(\point,\dpoint^{+})
	+ \braket{\dpoint^{+}-\dpoint}{\point - \base}.
\end{equation}
\end{lemma}

\begin{proof}
By definition, we get:
\begin{equation}
\begin{aligned}
\fench(\base,\dpoint^{+})
	&= \hreg(\base) +\hreg^{\ast}(\dpoint^{+}) - \braket{\dpoint^{+}}{\base}
	\\
\fench(\base,\dpoint)\hphantom{'}
	&= \hreg(\base) + \hreg^{\ast}(\dpoint) - \braket{\dpoint}{\base}.
	\\
\end{aligned}
\end{equation}
Then, by subtracting the above we get:
\begin{align}
\fench(\base,\dpoint^{+})-\fench(\base,\dpoint)
	&= \hreg(\base)
		+ \hreg^{\ast}(\dpoint^{+})
		- \braket{\dpoint^{+}}{\base}
		- \hreg(\base)
		- \hreg^{\ast}(\dpoint)
		+ \braket{\dpoint}{\base}
	\notag\\
	&=\hreg^{\ast}(\dpoint^{+})
		- \hreg^{\ast}(\dpoint)
		- \braket{\dpoint^{+} - \dpoint}{\base}
	\notag\\
	&=\hreg^{\ast}(\dpoint^{+})
		- \braket{\dpoint}{\mirror(\dpoint)}
		+ \hreg(\mirror(\dpoint))
		- \braket{\dpoint^{+} - \dpoint}{\base}
	\notag\\
	&=\hreg^{\ast}(\dpoint^{+})
		- \braket{\dpoint}{\point}
		+ \hreg(\point)
		- \braket{\dpoint^{+} - \dpoint}{\base}
	\notag\\
	&= \hreg^{\ast}(\dpoint^{+})
		+ \braket{\dpoint^{+} - \dpoint}{\point}
		- \braket{\dpoint^{+}}{\point}
		+ \hreg(\point)
		- \braket{\dpoint^{+} - \dpoint}{\base}
	\notag\\
	&= \fench(\point,\dpoint^{+})
		+ \braket{\dpoint^{+} - \dpoint}{\point - \base}
\end{align}
and our proof is complete.
\end{proof}

\subsection{Numerical sequence inequalities}

\label{app:sequences}


In this section, we provide some necessary inequalities on numerical sequences that we require for the convergence rate analysis of the previous sections.
Most of the lemmas presented below already exist in the literature, and go as far back as \citet{ACBG02} and \citet{MS10};
when appropriate, we note next to each lemma the references with the statement closest to the precise version we are using in our analysis.


\smallskip


\begin{lemma}[\citealp{MS10}, \citealp{LYC18}]
\label{lem:sqrt}
    For all non-negative numbers $\real_{1},\dotsc \real_{\run}$, the following inequality holds:
    \begin{equation}
        \sqrt{\denom + \sum_{\run=\start}^{\nRuns}\real_{\run}}\leq \denomsqrt + \sum_{\run=\start}^{\nRuns}\dfrac{\real_{\run}}{\sqrt{\sum_{\runalt=\start}^{\run}\real_{\runalt}}}\leq 2\sqrt{\denom +  \sum_{\run=\start}^{\nRuns}\real_{\run}}
    \end{equation}
\end{lemma}

\section{Analysis and proofs of the main results}
\label{app:proofs}


\para{The proof of the template inequality} We first prove the template inequality of \undergrad; this is the primary element of our proof of \cref{thm:undergrad}:
\template*

\begin{proof}
First, we set $\tilde{\dstate}_{\run}=\learn_{\run}{\dstate}_{\run}$. For all $\point \in \points$ we have:
\begin{align}
	&\curr[\step]\braket{\signal_{\run+1/2}}{\next[\state]-\point} \notag\\
	= &
	\braket{\frac{1}{\learn_{\run}}\tilde{\dstate}_{\run} - \frac{1}{\learn_{\run+1}}\tilde{\dstate}_{\run+1}}{\next[\state]-\point} \notag\\
	= &\frac{1}{\learn_{\run}}\braket{\tilde{\dstate}_{\run} - \tilde{\dstate}_{\run+1}}{\next[\state]-\point}+\left[\frac{1}{\learn_{\run+1}}-\frac{1}{\learn_{\run}} \right]\braket{0-\tilde{\dstate}_{\run+1}}{\next[\state]-\point} \notag\\
	= &\frac{1}{\eta_{\run}}\left[ \fench(\point, \tilde{\dstate}_{\run})-\fench(\point,\tilde{\dstate}_{\run+1})-\fench(\next[\state],\tilde{\dstate}_{\run})\right]
	\notag\\
	&\qquad
	+\left[\frac{1}{\learn_{\run+1}}-\frac{1}{\learn_{\run}} \right]\bigg( \fench(\point,0)-\fench(\point,\tilde{\dstate}_{\run+1})-\fench(\next[\state],0) \bigg)
	\tag*{\#\itshape\; from \cref{lem:3points}}\\
	\le &\frac{1}{\learn_{\run}}\fench(\point,\tilde{\dstate}_{\run})-\frac{1}{\learn_{\run+1}}\fench(\point,\tilde{\dstate}_{\run+1}) +  \left[\frac{1}{\learn_{\run+1}}-\frac{1}{\learn_{\run}} \right]\range  -\frac{1}{\learn_{\run}}\fench(\next[\state],\tilde{\dstate}_{\run}). \label{eq:first}
\end{align}
Here, the last inequality comes from the facts that $\fench(\point,0)=\hreg(\point)-\hreg(\mirror(0))=\hreg(\point)-\min_{\point \in \points}\hreg \le \range$ and $\fench(\cdot,\cdot)\geq 0$ and that $\learn_{\run}$ is decreasing.

As a consequence of \eqref{eq:first}, we have:
\begin{align}
\label{eq:ineq-update}
\curr[\step]\braket{\signal_{\run+1/2}}{\state_{\run+1/2}-\point}
	\notag\\
	&= \curr[\step]\braket{\signal_{\run+1/2}}{\state_{\run+1/2}-\next[\state]}+\curr[\step]\braket{\signal_{\run+1/2}}{\next[\state]-\point}
	\notag\\
	&\leq \curr[\step] \braket{\signal_{\run+1/2}}{\state_{\run+1/2}-\next[\state]}+\frac{1}{\learn_{\run}}\fench(\point, \tilde{\dstate}_{\run})-\frac{1}{\learn_{\run+1}}\fench(\point,\tilde{\dstate}_{\run+1})
	\notag\\
	&\qquad
	+ \left[\frac{1}{\learn_{\run+1}}-\frac{1}{\learn_{\run}} \right]\range-\frac{1}{\learn_{\run}}\fench(\next[\state],\tilde{\dstate}_{\run})	
\end{align}

On the other hand, if we let $\lead[\tilde{\dstate}] \defeq \curr[\learn] \lead[\dstate] $, we have
\begin{align}
\curr[\step] \braket{\signal_{\run}}{\state_{\run+1/2}-\next[\state]}
	&= \frac{1}{\curr[\learn]} \braket{\tilde{\dstate}_{\run}-\tilde{\dstate}_{\run+1/2}}{\state_{\run+1/2}-\next[\state]}
	\notag\\
	&= \frac{1}{\curr[\learn]} \bracks*{\fench(\next[\state],\tilde{\dstate}_{\run})-\fench(\next[\state],\tilde{\dstate}_{\run+1/2})-\fench(\state_{\run+1/2},\tilde{\dstate}_{\run})}
	\notag\\
\shortintertext{so}
\frac{1}{\learn_{\run}}\fench(\next[\state],\tilde{\dstate}_{\run})
	&=  \curr[\step]\braket{\signal_{\run}}{\state_{\run+1/2}-\state^{\run+1}}
	\notag\\
	&+\frac{1}{\learn_{\run}}\fench(\next[\state],\tilde{\dstate}_{\run+1/2})+\frac{1}{\learn_{\run}}\fench(\state_{\run+1/2},\tilde{\dstate}_{\run})  \label{eq:ineq-leading}.
\end{align}
Thus, replacing \eqref{eq:ineq-leading} into \eqref{eq:ineq-update}, we get:
\begin{align}
\curr[\step]\braket{\signal_{\run+1/2}}{\state_{\run+1/2}-\point}
	&\leq \curr[\step]\braket{\signal_{\run+1/2}-\signal_{\run}}{\state_{\run+1/2}-\next[\state]}
	\notag\\
	&+ \frac{1}{\learn_{\run}}\fench(\point,\tilde{\dstate}_{\run})-\frac{1}{\learn_{\run+1}}\fench(\point,\tilde{\dstate}_{\run+1})
	-\frac{1}{\learn_{\run}}\fench(\next[\state],\tilde{\dstate}_{\run+1/2})
	\notag\\
	&- \frac{1}{\learn_{\run}}\fench(\state_{\run+1/2},\tilde{\dstate}_{\run}) +\left[\frac{1}{\learn_{\run+1}}-\frac{1}{\learn_{\run}} \right]\range. \label{eq:eq}
\end{align}

Now, recall the definitions $\lead[\state] = \mirror(\lead[\tilde{\dstate}])$ and $\next[\state]  =\mirror (\next[\dstate])$ in \cref{alg:undergrad}, apply \cref{lem:Fenchel} to $\fench(\next[\state],\tilde{\dstate}_{\run+1/2})$ and $\fench(\state_{\run+1/2},\tilde{\dstate}_{\run})$ then combine with \eqref{eq:eq}, we get:
\begin{multline}
 \curr[\step]\braket{\signal_{\run+1/2}}{\state_{\run+1/2}-\point}\leq \curr[\step]\braket{\signal_{\run+1/2}-\signal_{\run}}{\state_{\run+1/2}-\next[\state]}+\frac{1}{\learn_{\run}}\fench(\point,\tilde{\dstate}_{\run})-\frac{1}{\learn_{\run+1}}\fench(\point,\tilde{\dstate}_{\run+1})
\\
-\frac{\hstr}{2\learn_{\run}}\norm{\next[\state]-\state_{\run+1/2}}^{2}-\frac{\hstr}{2\learn_{\run}}\norm{\state_{\run+1/2}-\state_{\run}}^{2} +\bigg( \frac{1}{\learn_{\run+1}}-\frac{1}{\learn_{\run}}\bigg)\range
\end{multline}
Hence, after telescoping and recalling the notation $\regstoch_{\nRuns}(\point) \defeq  \sum_{\run=1}^{\nRuns} {\curr[\step]}\inner*{\lead[\signal] }{\lead-\point}$, we get:
\begin{align}
\regstoch_{\nRuns}(\point)
	&\leq  \frac{1}{\init[\learn]}\fench(\point,\tilde{\dstate}_1) + \bigg( \frac{1}{\learn_{\nRuns+1}}-\frac{1}{\learn_{1}}\bigg)\range
	\notag\\
	&+ \sum_{\run=1}^{\nRuns}\curr[\step]\braket{\signal_{\run+1/2}-\signal_{\run}}{\state_{\run+1/2}-\next[\state]}
	\notag\\
	&-\sum_{\run=1}^{\nRuns}\frac{\hstr}{2\learn_{\run}}\norm{\next[\state]-\state_{\run+1/2}}^{2}-\sum_{\run=1}^{\nRuns}\frac{\hstr}{2\learn_{\run}}\norm{\state_{\run+1/2}-\state_{\run}}^{2}\label{eq:end_template}
\end{align}
Finally, by our initial choice of $\init[\dstate]=0$, we have $\fench(\point,\tilde{\dstate}_{\start})=\hreg(\point)-\min_{\point \in \points}\hreg(\point) \le \range$
and \eqref{eq:Propo_energy_inequ} follows \eqref{eq:end_template}. This concludes the proof of \cref{prop:template}.
\end{proof}

\para{Regret-to-rate conversion lemma} The next element in our proof is the following lemma that will be used to connect the term $\regstoch_{\nRuns}(\point)$ (which, in intuition, is similar to a regret term) and the term $\ex \bracks*{ \obj\parens{ \lastlead[\avg\state] } - \min\obj}$ whose bounds will characterize the convergence rate of \undergrad.

\begin{lemma}
\label{lem:regret}
	For any $\sol \in \sols$, for any $\nRuns$, we have:
\begin{align}
\exof*{ \obj\parens{ \lastlead[\avg\state] } - \min\obj	}
	&\leq	\exof*{\frac{2}{{\nRuns^{2}}} {\sum_{\run=1}^{\nRuns}{\curr[\step]}\inner{\nabla \obj \parens*{\lead[\avg\state]}}{\lead-\sol} }}
	\notag\\
	&= \frac{2}{\nRuns^2} \exof*{\regstoch_{\nRuns}(\sol)}.
\end{align}
\end{lemma}

\begin{remark}
A variation of \cref{lem:regret} appears in \cite{Cut19,KLBC19}; for the sake of completeness, we provide its proof below.
\endenv
\end{remark}

\begin{proof}
	Let us denote $\curr[H] := \sum_{\runalt =1}^{\run} \iter[\step]$. From the update rule of Algorithm~\ref{alg:undergrad}, we can rewrite $\lead = \frac{\curr[H] }{\curr[\step]} \lead[\avg\state] - \frac{\prev[H] }{\curr[\step]} {\avg\state}_{\run-1/2} $. Therefore,
	
\begin{align}
\lead - \sol
	&=\frac{\curr[H] }{\curr[\step]} \parens{\lead[\avg\state] - \sol} - \frac{\prev[H] }{\curr[\step]} \parens{{\avg\state}_{\run-1/2} -\sol}
	\notag\\
	&= \frac{1}{\curr[\step]} \bracks*{ \curr[\step] \parens{ \lead[\avg\state] - \sol  } + \prev[H] \parens{\lead[\avg\state] - {\avg\state}_{\run-1/2} }	}.
\end{align}
As a consequence, we have:
\begin{align}
\sum_{\run=1}^{\nRuns}{\curr[\step]}\inner{\nabla \obj(\lead[\avg\state])}{\lead - \sol} 
	&= \sum_{\run=1}^{\nRuns} \curr[\step] \inner*{\nabla \obj(\lead[\avg\state])}{\lead[\avg\state] -\sol}
	\notag\\
	&\qquad
	+ \prev[H] \sum_{\run=1}^{\nRuns} \inner*{\nabla \obj(\lead[\avg\state])}{\lead[\avg\state] -{\avg\state}_{\run-1/2}}
	\notag\\
	\ge &\sum_{\run=1}^{\nRuns}  \curr[\step] \bracks*{\obj(\lead[\avg\state]) - \obj(\sol) }
	\notag\\
	&\qquad
	+ \sum_{\run=1}^{\nRuns} \prev[H] \bracks*{\obj(\lead[\avg\state]) - \obj({\avg\state}_{\run-1/2}) } \notag\\
	&= \sum_{\run=1}^{\nRuns}
		\curr[\step]\bracks*{\obj(\lead[\avg\state]) - \obj(\sol) }
	\notag\\
	&\qquad
	+ \sum_{\run=1}^{\nRuns-1} \curr[\step] \bracks*{ \obj({\avg\state}_{\nRuns+1/2}) - \obj(\lead[\avg\state])}	
	\tag*{\#\itshape\; since $\curr[H] - \prev[H] = \curr[\step]$}
	\notag\\
	%
	%
	=&\bracks*{\obj({\avg\state}_{\nRuns+1/2}) - \obj(\sol)} \sum_{\run=1}^{\nRuns} \curr[\step]. \label{eq:app_adaEW_LemKavis}
	\end{align}
	Divide two sides of \eqref{eq:app_adaEW_LemKavis} by $\curr[H]>0$ and choose $\curr[\step]$ such that $\curr[H] > \frac{\nRuns^2}{2} $ (for example, choose $\curr[\step] = \step$), we obtain that: 
	\begin{align}
	\frac{2}{\nRuns^2} \sum_{\run=1}^{\nRuns}{\curr[\step]}\inner{\nabla \obj(\lead[\avg\state])}{\lead - \sol} \ge \obj({\avg\state}_{\nRuns+1/2}) - \obj(\sol) = \obj({\avg\state}_{\nRuns+1/2}) - \min \obj. \label{eq:app_adaEW_LemKavis2}
	\end{align}

Finally, we recall that by definition, $\lead[\signal] = \orcl(\lead[\avg\state];\lead[\sample])
= \nabla\obj(\lead[\avg\state]) + \lead[\noise]$ where $\exof*{\lead[\noise] \given \lead[\filter] } = 0$. Therefore, by the law of total expectation, we have:
\begin{align}
\regstoch_{\nRuns}(\sol)
	&= \exof*{\sum_{\run=1}^{\nRuns}{\curr[\step]}\inner{\nabla \obj(\lead[\avg\state])}{\lead - \sol}} + \exof*{\sum_{\run=1}^{\nRuns}{\curr[\step]}\inner{\lead[\noise]}{\lead - \sol}}
	\notag\\
	&= \exof*{\sum_{\run=1}^{\nRuns}{\curr[\step]}\inner{\nabla \obj(\lead[\avg\state])}{\lead - \sol}}
	\notag\\
	&+ \exof*{\sum_{\run=1}^{\nRuns}{\curr[\step]} \exof*{\inner{\lead[\noise]}{\lead - \sol} \given \lead[\filter]}}
	\notag\\
	= &\exof*{\sum_{\run=1}^{\nRuns}{\curr[\step]}\inner{\nabla \obj(\lead[\avg\state])}{\lead - \sol}}. \label{eq:law_expec}
\end{align}
Then, taking expectations on both sides of \eqref{eq:app_adaEW_LemKavis2} and invoking \eqref{eq:law_expec} concludes the proof.
\end{proof}

\para{Proof of \eqref{eq:rate-under-BG}: convergence of \undergrad under \eqref{eq:LC}/\eqref{eq:BG}}

Our starting point is \cref{eq:Propo_energy_inequ} that we established in \cref{prop:template} that leads to the following inequality:

\begin{equation}
	\regstoch_{\nRuns}(\point) \leq	\frac{\range}{\learn_{\nRuns+1}}
	+ 
	\sum_{\run=1}^{\nRuns}\curr[\step]\braket{\signal_{\run+1/2}-\signal_{\run}}{\state_{\run+1/2}-\next[\state]}
	-
	\sum_{\run=1}^{\nRuns}\frac{\hstr}{2\learn_{\run}}\norm{\next[\state]-
	\state_{\run+1/2}}^{2}
	\label{eq:BG_first}
\end{equation}

 We now focus on the second term in the right hand side of \eqref{eq:BG_first}. From the Cauchy-Schwarz inequality and the fact that
\begin{equation}
\dnorm*{\dstate-\alt\dstate} \norm*{\state-\statealt}
	= \min_ {\real>0}\braces*{  \frac{1}{2\real}\dnorm*{\dstate-\alt\dstate} ^{2}	+\frac{\real}{2}\norm*{\state-\statealt}^{2}}
\end{equation}
for any $\state, \statealt, \dstate,\alt\dstate \in \mathbb{R}^{\dims}$,\footnote{This can be proved trivially:  $\real^* =\dnorm*{\dstate-\alt\dstate} \norm*{\state-\statealt} $ is a minimizer of the function \mbox{$\psi(\real)=\frac{1}{2\real}\dnorm*{\dstate - \alt\dstate}^{2} + \frac{\real}{2}\norm*{\state-\statealt}^{2}$}.} we have:
 
\begin{align}
	\sum_{\run=1}^{\nRuns}\curr[\step]\braket{\signal_{\run+1/2}-\signal_{\run}}{\state_{\run+1/2}-\next[\state]} 
	&\le  \sum_{\run=1}^{\nRuns}\curr[\step] \dnorm{\signal_{\run+1/2}-\signal_{\run}} \norm{\state_{\run+1/2}-\next[\state]} \notag\\
	&\le  \frac{1}{2 \hstr}\sum_{\run=1}^{\nRuns} \curr[\step]^{2} \learn_{\run+1}\dnorm{\signal_{\run+1/2}-\signal_{\run}}^{2}
	\notag\\
	&+  \frac{\hstr}{2 } \sum_{\run=1}^{\nRuns}\frac{1}{\learn_{\run+1}} \norm{\next[\state]-\state_{\run+1/2}}^{2}. \label{eq:BG_2}
\end{align}

Moreover, from the definition of $\next[\learn]$ and by applying \cref{lem:sqrt}, we have:
\begin{align}
	\frac{1}{2 \hstr}\sum_{\run=1}^{\nRuns} \curr[\step]^{2} \learn_{\run+1}\dnorm{\signal_{\run+1/2}-\signal_{\run}}^{2} 
	&
	= \frac{\numer}{2\hstr}\sum_{\run=1}^{\nRuns}\frac{\curr[\step]^{2}\dnorm{\signal_{\run+1/2}-
	\signal_{\run}}^{2}}{\sqrt{\denom+\sum_{\runalt=1}^{\run}\dnorm{\signal_{\runalt+1/2} -
	\signal_{\runalt}}^{2}}}
	\notag\\
	&\leq \frac{\numer}{\hstr} \sqrt{\denom+\sum_{\run=1}^{\nRuns}\curr[\step]^{2}\dnorm{\signal_{\run+1/2}-\signal_{\run}}^{2}} - \frac{\numer \sqrt{\denom}}{2\hstr}
	\notag\\
	&=\frac{\numer^2}{\hstr \cdot \learn_{\nRuns+1}} - \frac{\numer \sqrt{\denom}}{2\hstr}. \label{eq:BG_3}
\end{align}

Combine \eqref{eq:BG_2} and \eqref{eq:BG_3} with \eqref{eq:BG_first} and by the compactness of the feasible region $\points$, we get:
\begin{align}
\regstoch_{\nRuns}(\point)
	\leq
	&\frac{\range}{\learn_{\nRuns+1}}
	+ \frac{1}{2\hstr}\sum_{\run=1}^{\nRuns}\curr[\step]^{2}\learn_{\run+1}\dnorm{\signal_{\run+1/2}-\signal_{\run}}^{2}
	\notag\\
	&\hphantom{\frac{\range}{\learn_{\nRuns+1}}}
	+\frac{\hstr}{2}\sum_{\run=1}^{\nRuns}\left[\frac{1}{\learn_{\run+1}}-\frac{1}{\learn_{\run}} \right]\norm{\next[\state]-\state_{\run+1/2}}^{2} \notag\\
	\leq
	&\frac{\range}{\learn_{\nRuns+1}}
	+ \frac{\numer^2}{\hstr \cdot \learn_{\nRuns+1}} - \frac{\numer \sqrt{\denom}}{2\hstr} + \frac{\hstr \diampoints^{2}}{2}\sum_{\run=1}^{\nRuns}\left[\frac{1}{\learn_{\run+1}}-\frac{1}{\learn_{\run}} \right] \notag\\
	= &\frac{1}{\learn_{\nRuns+1}} \parens*{\range + \frac{\numer^2}{\hstr}  + \frac{\hstr \diampoints^{2}}{2 } } -  \numer \sqrt{\denom} \parens*{\frac{1}{2\hstr} + \frac{\hstr \diampoints^2}{2} }.
\label{eq:BG_second}
\end{align}

Hence, by invoking \cref{lem:regret}, we have:
\begin{align}
\exof*{\obj(\overline{\state}_{\nRuns+1/2})-\min_{\point \in \points}\obj(\point)}
	&\leq \frac{2 \exof*{ \regstoch_{\nRuns}(\sol)}}{\nRuns^2}
	\notag\\
	&= \frac{2}{\nRuns^2} \bracks*{\parens*{\range + \frac{\numer^2}{\hstr}  + \frac{\hstr \diampoints^{2}}{2 } }}
	\exof*{\frac{1}{\learn_{\nRuns+1}}}
	\notag\\
	&-  \frac{2\numer \sqrt{\denom}}{\nRuns^2} \parens*{\frac{1}{2\hstr} + \frac{\hstr \diampoints^2}{2} } . \label{eq:1}
\end{align}
On the other hand, by the definition of $\learn_{\nRuns+1}$, we get:
\begin{align}
	\exof*{\frac{1}{\learn_{\nRuns+1}}} &\leq  \frac{1}{\numer}
	\ex \left[\sqrt{\denom + \sum_{\run=1}^{\nRuns}\curr[\step]^{2}\dnorm{\signal_{\run+1/2}-\signal_{\run}}^{2}} \right]
	%
	%
	\leq  \frac{1}{\numer}\sqrt{\denom + \sum_{\run=1}^{\nRuns}\curr[\step]^{2}\exof*{\dnorm{\signal_{\run+1/2}-\signal_{\run}}^{2}}}  .
\end{align}
Moreover, we have that:
\begin{align}
	\exof*{ \dnorm{\signal_{\run+1/2}-\signal_{\run}}^{2} }
	&= \exof*{\dnorm{\nabla \obj(\overline{\state}_{\run+1/2})-\nabla \obj(\overline{\state}_{\run})-(\noise_{\run+1/2}-\noise_{\run})}^{2}}
	\notag\\
	&\leq \exof*{2\dnorm{\nabla \obj(\overline{\state}_{\run+1/2})-\nabla \obj(\overline{\state}_{\run})}^{2}+2\dnorm{\noise_{\run+1/2}-\noise_{\run}}^{2}}
	\notag\\
	&\leq \exof*{4(\dnorm{\nabla \obj(\overline{\state}_{\run+1/2})}^{2}+\dnorm{\nabla \obj(\overline{\state}_{\run})}^{2})+
	4(\dnorm{\noise_{\run+1/2}}^{2}+\dnorm{\noise_{\run}}^{2})}
	\notag\\
	&\leq \exof*{8\gbound^{2}+4(\dnorm{\noise_{\run+1/2}}^{2}+\dnorm{\noise_{\run}}^{2})}
	\tag*{\#\itshape\; from \eqref{eq:BG}}\\
	&=  8\left[\gbound^{2}+\sigma^{2} \right]. 
\end{align}

Therefore, when we choose the step-size parameters $\curr[\step] = \run, \forall \run$ as indicated in \cref{thm:undergrad}, we have:
\begin{align}
 \parens*{\range + \frac{\numer^2}{\hstr}  + \frac{\hstr \diampoints^{2}}{2 } } \exof*{\frac{1}{\learn_{\nRuns+1}}}  
	\leq &\parens*{\range + \frac{\numer^2}{\hstr}  + \frac{\hstr \diampoints^{2}}{2 } }  \frac{1}{\numer}\sqrt{\denom +  8\left(\gbound^{2}+\sigma^{2} \right) \sum_{\run=1}^{\nRuns}\curr[\step]^{2}} \notag\\
	\le &\parens*{\frac{\range}{\numer} + \frac{\numer}{\hstr}  + \frac{\hstr \diampoints^{2}}{2 \numer} } \sqrt{\denom + 8 \parens*{\gbound^2 + \sigma^2} \nRuns^{3}} .\label{eq:BG_conclude_1}
\end{align}
Finally, substituting \eqref{eq:BG_conclude_1} into \eqref{eq:1}, we get:
\begin{align}
\exof*{\obj(\overline{\state}_{\nRuns+1/2})-\min_{\point \in \points}\obj(\point)}
	&\leq 2\parens*{\frac{\range}{\numer} + \frac{\numer}{\hstr}  + \frac{\hstr \diampoints^{2}}{2 \numer} } \frac{\sqrt{\denom + 8 \parens*{\gbound^2 + \sigma^2} \nRuns^{3}}}{\nRuns^2}
	\notag\\
	&- \frac{\numer \sqrt{\denom}}{\nRuns^2} \parens*{\frac{1}{2\hstr} + \frac{\hstr \diampoints^2}{2}}.
\end{align}
Then, from our choice for $\numer$ and $\denom$ in \cref{thm:undergrad}, we obtain:
\begin{equation}
	\ex\left[\obj(\overline{\state}_{\nRuns+1/2})-\min_{\point \in \points}\obj(\point) \right] \leq 
	2 \frac{\hconst}{\sqrt{\hstr}} \frac{\sqrt{\hstr + 8 \parens*{\gbound^2 + \sigma^2}}}{\sqrt{\nRuns}}.
\end{equation}

\para{Proof of \eqref{eq:rate-under-LG}: convergence of \undergrad under \eqref{eq:LG}/\eqref{eq:LS}}

From \eqref{eq:Propo_energy_inequ} and \eqref{eq:BG_2}, we have:
\begin{align}
\regstoch_{\nRuns}(\point) 
	\leq 
	\frac{\range}{\learn_{\nRuns+1}}
	&+	\frac{1}{2 \hstr}\sum_{\run=1}^{\nRuns} \curr[\step]^{2} \learn_{\run+1}\dnorm{\signal_{\run+1/2}-\signal_{\run}}^{2}
	\notag\\
	&+  \frac{\hstr}{2 } \sum_{\run=1}^{\nRuns} \parens*{\frac{1}{\next[\learn]} - \frac{1}{\curr[\learn]}} \norm{\next[\state]-\state_{\run+1/2}}^{2}
	%
	- \sum_{\run=\start}^{\nRuns}\frac{\hstr}{2\curr[\learn]} { \norm{\lead - \curr[\state]}^{2}}. \label{eq:LG_first}
 	%
 	%
\end{align}

We analyze the terms in the right-hand-side of \eqref{eq:LG_first}. First, we have:
\begin{align}
	\frac{\hstr}{2 } \sum_{\run=1}^{\nRuns} \parens*{\frac{1}{\next[\learn]} - \frac{1}{\curr[\learn]}} \norm{\next[\state]-\state_{\run+1/2}}^{2}  \le	\frac{\hstr \diampoints^2}{2} \parens*{\frac{1}{\learn_{\nRuns+1}} - \frac{1}{\learn_{\start}} }.\label{eq:LG_1}
\end{align}

Second, we have:
\begin{align}
	\frac{\hstr}{2}\sum_{\run=1}^{\nRuns}\frac{1}{\learn_{\run+1}}\norm{\state_{\run+1/2}-\state_{\run}}^{2}&=
	\frac{\hstr}{2}\sum_{\run=1}^{\nRuns}\left[\frac{1}{\learn_{\run+1}}-\frac{1}{\learn_{\run}}\right]\norm{\state_{\run+1/2}-\state_{\run}}^{2}
	\notag\\
	&\qquad
	+\frac{\hstr}{2}\sum_{\run=1}^{\nRuns}\frac{1}{\learn_{\run}}\norm{\state_{\run+1/2}-\state_{\run}}^{2}
	\notag\\
	&\leq \frac{\hstr\diampoints^{2}}{2} \parens*{\frac{1}{\learn_{\nRuns+1}} - \frac{1}{\learn_{\start}}} + \frac{\hstr}{2}\sum_{\run=1}^{\nRuns}\frac{1}{\learn_{\run}}\norm{\state_{\run+1/2}-\state_{\run}}^{2}
\end{align}

Hence, 
\begin{equation}
	- \frac{\hstr}{2}\sum_{\run=1}^{\nRuns}\frac{1}{\learn_{\run}}\norm{\state_{\run+1/2}-\state_{\run}}^{2} \le \frac{\hstr\diampoints^{2}}{2} \parens*{\frac{1}{\learn_{\nRuns+1}} - \frac{1}{\learn_{\start}}}  -  \frac{\hstr}{2}\sum_{\run=1}^{\nRuns}\frac{1}{\learn_{\run+1}}\norm{\state_{\run+1/2}-\state_{\run}}^{2}. \label{eq:LG_2}
\end{equation}

Combine \eqref{eq:LG_1} and \eqref{eq:LG_2} with \eqref{eq:LG_first}, we have:
\begin{align}
	\regstoch_{\nRuns}(\point)
	&\leq \frac{\range +	\hstr\diampoints^{2}}{\learn_{\nRuns+1}}	- \frac{  \hstr\diampoints^{2}}{\learn_{\start}}
	\notag\\
	&+\frac{1}{2\hstr}\sum_{\run=1}^{\nRuns}\curr[\step]^{2}\learn_{\run+1}\dnorm{\signal_{\run+1/2}-\signal_{\run}}^{2}
	-\frac{\hstr}{2}\sum_{\run=1}^{\nRuns}\frac{1}{\learn_{\run+1}}\norm{\state_{\run+1/2}-\state_{\run}}^{2}. \label{eq:equation-3}
\end{align}

We will analyze the terms in the right-hand-side of \eqref{eq:equation-3}. To do this, we first introduce the quantities
\begin{subequations}
\begin{align}
\mindiff_{\run}^{2}
	&= \min\{\dnorm{\nabla \obj(\overline{\state}_{\run+1/2})-\nabla \obj(\overline{\state}_{\run})}^{2},\dnorm{\signal_{\run+1/2}-\signal_{\run}}^{2}\}
\shortintertext{and}
\diff_{\run}
	&= \left[\signal_{\run+1/2}-\signal_{\run}\right]-\left[ \nabla \obj(\overline{\state}_{\run+1/2})-\nabla \obj(\overline{\state}_{\run})\right].
\end{align}
\end{subequations}
We also define
\begin{equation}
	\tilde{\learn}_{\run}=\frac{\numer}{\sqrt{\denom + \sum_{\runalt=1}^{\run-1}\alpha_{\runalt}^{2}\mindiff_{\runalt}^{2}}}.
\end{equation}


By these definitions, we obtain that
\begin{align}
	\dnorm{\signal_{\run+1/2}-\signal_{\run}}^{2}
	%
	&\leq \mindiff_{\run}^{2}+\left[ \dnorm{\signal_{\run+1/2}-\signal_{\run}}^{2}-\min\{ \dnorm{\nabla \obj(\overline{\state}_{\run+1/2})-\nabla \obj(\overline{\state}_{\run})}^{2}, \dnorm{\signal_{\run+1/2}-\signal_{\run}}^{2}\} \right]
	\notag\\
	&\leq \mindiff_{\run}^{2}+\max \{0, \dnorm{\signal_{\run+1/2}-\signal_{\run}}^{2}-\dnorm{\nabla \obj(\overline{\state}_{\run+1/2})-\nabla \obj(\overline{\state}_{\run})}^{2}\}
	\notag\\
	&\leq \mindiff_{\run}^{2}+\mindiff_{\run}^{2}+2\dnorm{\diff_{\run}}^{2}
	\notag\\
	&=2\mindiff_{\run}^{2}+2\dnorm{\diff_{\run}}^{2}. \label{eq:inequa}
\end{align}

Here, the last inequality is obtained by the fact that if $\dnorm{\signal_{\run+1/2}-\signal_{\run}}^{2}\geq \dnorm{\nabla \obj(\overline{\state}_{\run+1/2})-\nabla \obj(\overline{\state}_{\run})}^{2}$ then it yields:
\begin{equation}
	\dnorm{\signal_{\run+1/2}-\signal_{\run}}^{2}-\dnorm{\nabla \obj(\overline{\state}_{\run+1/2})-\nabla \obj(\overline{\state}_{\run})}^{2}\leq \mindiff_{\run}^{2}+2\dnorm{\diff_{\run}}^{2}.
\end{equation}


Therefore, we have:
\begin{align}
	\frac{1}{2\hstr}\sum_{\run=1}^{\nRuns}\curr[\step]^{2}\learn_{\run+1}\dnorm{\signal_{\run+1/2}-\signal_{\run}}^{2} 
	&=  \frac{\numer}{2\hstr}\sum_{\run=1}^{\nRuns} \frac{\curr[\step]^{2} \dnorm{\signal_{\run+1/2}-\signal_{\run}}^{2}}{ \sqrt{\denom + \sum_{\runalt =1}^{\run} \iter[\step]^2 \norm*{\signal_{\runalt+{1/2}} - \iter[\signal] }^2	}} \notag\\
	&\leq \frac{\numer}{\hstr }\sqrt{\denom +\sum_{\run=1}^{\nRuns}\curr[\step]^{2}\dnorm{\signal_{\run+1/2}-\signal_{\run}}^{2}} - \frac{\numer \sqrt{\denom}}{2 \hstr}
	\tag*{\#\itshape\; from \cref{lem:sqrt}}
	\notag\\
	&\leq\frac{\numer}{\hstr}\sqrt{\denom +2\sum_{\run=1}^{\nRuns}\curr[\step]^{2}\mindiff_{\run}^{2}+2\sum_{\run=1}^{\nRuns}\curr[\step]^{2}\dnorm{\diff_{\run}}^{2}} - \frac{\numer \sqrt{\denom}}{2 \hstr} 
	\tag*{\#\itshape\; from \eqref{eq:inequa}}\\
	&\leq \frac{\numer\sqrt{2}}{\hstr}
	\sqrt{\denom + \sum_{\run=1}^{\nRuns}\curr[\step]^{2}B^{2}_{\run}}
	+ \frac{\numer \sqrt{2}}{\hstr}\sqrt{\sum_{\run=1}^{\nRuns}\curr[\step]^{2}\dnorm{\diff_{\run}}^{2}} - \frac{\numer \sqrt{\denom}}{2 \hstr}
	\notag\\
	&\leq \frac{\numer\sqrt{2}}{\hstr} \parens*{\sqrt{\denom} + \sum_{\run=1}^{\nRuns}\frac{\curr[\step]^{2}\mindiff_{\run}^{2}}{ \sqrt{\denom + \sum_{\runalt=1}^{\run}\alpha_{\runalt}^{2}\mindiff_{\runalt}^{2}}} }
	\notag\\
	&\qquad
	+ \frac{\numer \sqrt{2}}{\hstr}\sqrt{\sum_{\run=1}^{\nRuns}\curr[\step]^{2}\dnorm{\diff_{\run}}^{2}} - \frac{\numer \sqrt{\denom}}{2 \hstr}
	\tag*{\#\itshape\; from \cref{lem:sqrt}}\\
	&=\frac{\sqrt{2}}{\hstr}\sum_{\run=1}^{\nRuns}\curr[\step]^{2}\tilde{\learn}_{\run+1}\mindiff_{\run}^{2}
	+ \frac{\numer \sqrt{2}}{\hstr}\sqrt{\sum_{\run=1}^{\nRuns}\curr[\step]^{2}\dnorm{\diff_{\run}}^{2}} + \frac{\numer \sqrt{\denom}}{\hstr} \parens*{\sqrt{2} - \frac{1}{2}}. \label{eq:LG_1_term}
\end{align}

On the other hand, by the update rule in \cref{alg:undergrad} and our choice of $\curr[\step] = \run, \forall \run$ as in \cref{thm:undergrad}, we also have $\curr[\state] - \lead[\state] = \frac{\sum_{\runalt=1}^{\run} \iter[\step]}{\curr[\step]} \parens*{ \curr[\avg\state] - \lead[\avg\state] } =  \frac{\next[\step]}{2} \parens*{ \curr[\avg\state] - \lead[\avg\state] } $. Use this and recall that $\frac{1}{\tilde{\learn_{\run}}}\leq \frac{1}{\learn_{\run}}$ for any $\run$ and that $\obj$ is $\lips$-smooth over $\points$, we have:
\begin{align}
- \frac{\hstr}{2}\sum_{\run=1}^{\nRuns}\frac{1}{\learn_{\run+1}}\norm{\state_{\run}-\state_{\run+1/2}}^{2}
	&\leq -\frac{\hstr}{2}\sum_{\run=1}^{\nRuns}\frac{1}{\tilde{\learn}_{\run+1}}\norm{\state_{\run}-\state_{\run+1/2}}^{2}
	\notag\\
	&= - \frac{\hstr}{8}\sum_{\run=1}^{\nRuns}\frac{\next[\step]^2}{\tilde{\learn}_{\run+1}}\norm{\curr[\avg\state]-\lead[\avg\state]}^{2}
	\notag\\
	&\le  - \frac{\hstr}{8}\sum_{\run=1}^{\nRuns}\frac{1}{\tilde{\learn}_{\run+1}}  \frac{1}{\lips^2} \dnorm{\nabla \obj(\overline{\state}_{\run}) - \nabla \obj(\overline{\state}_{\run+1/2})}^{2} \notag\\
	&\le - \frac{\hstr}{8\lips^2} \sum_{\run=1}^{\nRuns} \frac{\curr[\step]^2 \curr[\mindiff]^2 }{\next[\tilde{\learn}]}.\label{eq:LG_2_term}
\end{align}

Finally, letting $\hconst = \sqrt{\range + \hstr\diampoints^{2}}$, \eqref{eq:inequa} yields:
\begin{align}
	\frac{\hconst^{2}}{\learn_{\nRuns+1}}
	&= \frac{\hconst^{2}}{\numer}  \sqrt{\denom + \sum_{\run=1}^{\nRuns}\curr[\step]^{2}\dnorm{\signal_{\run+1/2}-\signal_{\run}}^{2}} 
	\notag\\
	&\leq\frac{ \hconst^{2}}{\numer} \sqrt{\denom + 2\sum_{\run=1}^{\nRuns}\curr[\step]^{2}\mindiff_{\run}^{2}+2\sum_{\run=1}^{\nRuns}\curr[\step]^{2}\dnorm{\diff_{\run}}^{2}} 
	\notag\\
	&\leq \frac{\sqrt{2}\hconst^{2}}{\numer}\sqrt{\denom + \sum_{\run=1}^{\nRuns}\curr[\step]^{2}\mindiff_{\run}^{2}}
	+
	\frac{\sqrt{2}\hconst^{2}}{\numer}\sqrt{\sum_{\run=1}^{\nRuns}\curr[\step]^{2}\dnorm{\diff_{\run}}^{2}} 
	\notag\\
	&\le  \frac{\sqrt{2}\hconst^{2}}{\numer} \parens*{\sqrt{\denom} + \sum_{\run=1}^{\nRuns} \frac{\curr[\step]^2 \curr[\mindiff]^2 }{ \sqrt{\denom + \sum_{\runalt=1}^{\run} \iter[\step]^2 \iter[\mindiff]^2  }} }
	+
	\frac{\sqrt{2}\hconst^{2}}{\numer}\sqrt{\sum_{\run=1}^{\nRuns}\curr[\step]^{2}\dnorm{\diff_{\run}}^{2}} 
	\notag\\
	&\leq \frac{\sqrt{2}\hconst^{2}}{\numer^2}
	\sum_{\run=1}^{\nRuns}\curr[\step]^{2}\tilde{\learn}_{\run+1}\mindiff_{\run}^{2}
	+ \frac{\sqrt{2}\hconst^{2} \sqrt{\denom}}{\numer} +
	\frac{\sqrt{2}\hconst^{2}}{\numer}\sqrt{\sum_{\run=1}^{\nRuns}\curr[\step]^{2}\dnorm{\diff_{\run}}^{2}} .\label{eq:LG_3_term}
\end{align}
Combine \eqref{eq:LG_1_term}, \eqref{eq:LG_2_term} and \eqref{eq:LG_3_term} into \eqref{eq:equation-3}, we have:

\begin{align}
	\regstoch_{\nRuns}(\point) 
		&\leq \sqrt{2} \sum_{\run=1}^{\nRuns}\curr[\step]^{2}\mindiff_{\run}^{2}\left[ \parens*{  \frac{\range}{\numer^2} + \frac{\hstr\diampoints^{2}}{\numer^2} + \frac{1}{\hstr} } \tilde{\learn}_{\run+1}-\frac{\hstr}{8\lips^{2}\tilde{\learn}_{\run+1}} \right]
	\notag\\
	&+ \sqrt{2} \parens*{\frac{\range}{\numer} + \frac{\hstr\diampoints^{2}}{\numer} + \frac{\numer}{\hstr}} \sqrt{\sum_{\run=1}^{\nRuns}\curr[\step]^{2}\dnorm{\diff_{\run}}^{2}}
	\notag\\
	&\qquad
	+ \bracks*{\frac{\numer \sqrt{\denom}}{\hstr} \parens*{\sqrt{2} - \frac{1}{2}} - \frac{\hstr \diampoints^2 \sqrt{\denom}}{\numer} + \frac{\sqrt{2}\hconst^{2} \sqrt{\denom}}{\numer}}. \label{eq:LG_3.5}
\end{align}
Now, define $\nRuns_0$ as follows:
\begin{equation}
	\nRuns_0=\max\braces*{1\leq \run \leq \nRuns: \tilde{\learn}_{\run+1}\geq \frac{\sqrt{\hstr}}{ \sqrt{8\lips^2 \parens*{ \frac{\range}{\numer^2} + \frac{\hstr\diampoints^{2}}{\numer^2} + \frac{1}{\hstr} }}}}
\end{equation}

In other words, for any $\run > \nRuns_0$, $\parens*{  \frac{\range}{\numer^2} + \frac{\hstr\diampoints^{2}}{\numer^2} + \frac{1}{\hstr} } \tilde{\learn}_{\run+1}-\frac{\hstr}{8\lips^{2}\tilde{\learn}_{\run+1}} <0$. As a consequence, 
\begin{align}
	&\sqrt{2} \sum_{\run=1}^{\nRuns}\curr[\step]^{2}\mindiff_{\run}^{2}\left[ \parens*{  \frac{\range}{\numer^2} + \frac{\hstr\diampoints^{2}}{\numer^2} + \frac{1}{\hstr} } \tilde{\learn}_{\run+1}-\frac{\hstr}{8\lips^{2}\tilde{\learn}_{\run+1}} \right] \notag\\
	=	&\sqrt{2} \sum_{\run=1}^{\nRuns_0}\curr[\step]^{2}\mindiff_{\run}^{2}\left[ \parens*{  \frac{\range}{\numer^2} + \frac{\hstr\diampoints^{2}}{\numer^2} + \frac{1}{\hstr} } \tilde{\learn}_{\run+1}-\frac{\hstr}{8\lips^{2}\tilde{\learn}_{\run+1}} \right]
	\notag\\
	\le &\sqrt{2}  \parens*{  \frac{\range}{\numer^2} + \frac{\hstr\diampoints^{2}}{\numer^2} + \frac{1}{\hstr} } \sum_{\run=1}^{\nRuns_0}\curr[\step]^{2}\mindiff_{\run}^{2} \tilde{\learn}_{\run+1}  \notag\\
	= &\sqrt{2}  \parens*{  \frac{\range}{\numer^2} + \frac{\hstr\diampoints^{2}}{\numer^2} + \frac{1}{\hstr} } \sum_{\run=1}^{\nRuns_0} \numer \frac{ \curr[\step]^{2}\mindiff_{\run}^{2}}{\sqrt{\denom + \sum_{\runalt=1}^{\run}\alpha_{\runalt}^{2}\mindiff_{\runalt}^{2}}}
	\notag\\
	\le &\sqrt{2}  \parens*{  \frac{\range}{\numer} + \frac{\hstr\diampoints^{2}}{\numer} + \frac{\numer}{\hstr} } \parens*{2 \sqrt{\denom + \sum_{\run=1}^{\nRuns_0} \curr[\step]^2 \curr[\mindiff]^2}  - \sqrt{\denom} }
	\notag\\
	= &2  \sqrt{2}  \parens*{  \frac{\range}{\numer} + \frac{\hstr\diampoints^{2}}{\numer} + \frac{\numer}{\hstr} } \frac{\numer}{\tilde{\learn}_{\nRuns_0+1} } - \sqrt{2 \denom}  \parens*{  \frac{\range}{\numer} + \frac{\hstr\diampoints^{2}}{\numer} + \frac{\numer}{\hstr} }
	\notag\\
	\le &8 \parens*{  \frac{\range}{\numer^2} + \frac{\hstr\diampoints^{2}}{\numer^2} + \frac{1}{\hstr} }^{3/2}  \frac{{\numer}^2 \lips }{\sqrt{\hstr}} - \sqrt{2 \denom}  \parens*{  \frac{\range}{\numer} + \frac{\hstr\diampoints^{2}}{\numer} + \frac{\numer}{\hstr} } \label{eq:LG_4}
	.
\end{align}

Combine \eqref{eq:LG_3.5} with \eqref{eq:LG_4} and use the fact that $\exof*{ \dnorm*{\diff}^2} \le 4 \sigma^2$, we have:

\begin{align}
\exof*{\regstoch_{\nRuns}(\point)}
	&\leq  8 \parens*{  \frac{\range}{\numer^2} + \frac{\hstr\diampoints^{2}}{\numer^2} + \frac{1}{\hstr} }^{3/2}  \frac{{\numer}^2 \lips }{\sqrt{\hstr}} - \sqrt{2 \denom}  \parens*{  \frac{\range}{\numer} + \frac{\hstr\diampoints^{2}}{\numer} + \frac{\numer}{\hstr} }
	\notag\\
	&+ 2\sqrt{2} \parens*{\frac{\range}{\numer} + \frac{\hstr\diampoints^{2}}{\numer} + \frac{\numer}{\hstr}} \sigma \sqrt{\sum_{\run=1}^{\nRuns}\curr[\step]^{2}}
	\notag\\
	&\qquad
	+ \bracks*{\frac{\numer \sqrt{\denom}}{\hstr} \parens*{\sqrt{2} - \frac{1}{2}} - \frac{\hstr \diampoints^2 \sqrt{\denom}}{\numer} + \frac{\sqrt{2}\hconst^{2} \sqrt{\denom}}{\numer}}
\end{align}

Recall the choice $\curr[\step] = \run, \forall \run$, apply \cref{lem:regret}, we have:


\begin{align}
\ex\left[\obj(\overline{\state}_{\nRuns+1/2})-\min_{\point \in \points}\obj(\point) \right]
	&\leq	\frac{16}{\nRuns^2} \parens*{  \frac{\range}{\numer^2} + \frac{\hstr\diampoints^{2}}{\numer^2} + \frac{1}{\hstr} }^{3/2}  \frac{{\numer}^2 \lips }{\sqrt{\hstr}}
	\notag\\
	&+ \frac{4\sqrt{2} }{\sqrt{\nRuns}} \parens*{\frac{\range}{\numer} + \frac{\hstr\diampoints^{2}}{\numer} + \frac{\numer}{\hstr}} \sdev
	+ \frac{2 \const(\numer, \denom)}{\nRuns^2}. \label{eq:last}
\end{align}
where we set
\begin{equation}
\const(\numer, \denom) \defeq \frac{\numer \sqrt{\denom}}{\hstr} \parens*{\sqrt{2} - \frac{1}{2}} - \frac{\hstr \diampoints^2 \sqrt{\denom}}{\numer} + \frac{\sqrt{2}\hconst^{2} \sqrt{\denom}}{\numer} - \sqrt{2 \denom}  \parens*{  \frac{\range}{\numer} + \frac{\hstr\diampoints^{2}}{\numer} + \frac{\numer}{\hstr} }.
\end{equation}
Finally, replace $\numer = \sqrt{\hstr \hconst^{2}}$ and $\denom = \hstr$ as chosen in \cref{thm:undergrad} into \eqref{eq:last} and note that with these choices, $\const(\numer, \denom) \le -\frac{1}{2} \hconst \sqrt{\hstr} \le 0$; we rewrite \eqref{eq:last} as follows:
\begin{equation}
	\ex\left[\obj(\overline{\state}_{\nRuns+1/2})-\min_{\point \in \points}\obj(\point) \right] \leq	\frac{32\sqrt{2} \lips}{\nRuns^2} \parens*{ \frac{\hconst^{2}}{ {\hstr}} }	
	+
	\frac{8\sqrt{2} \sigma }{\sqrt{\nRuns}}  \frac{\hconst}{\sqrt{\hstr}}.
\end{equation}

\para{Convergence of \undergrad in unbounded domains} Finally, we give the proof of \cref{thm:unbounded} concerning the deterministic \ac{SFO} in the \eqref{eq:LG} case with a possibly \emph{unbounded} domain $\points$.

\begin{proof}
Since the respective learning rate $\curr[\learn]$
  is non-increasing and non-negative, we have that its limit exists. Particularly,
\begin{equation}
\lim_{\run \to \infty} \eta_{\run}=\inf\nolimits_{\run \in \N} \eta_{\run}
	\geq 0
\end{equation}
Let us now assume that $\inf_{\run \in \N}\learn_{\run}=0$.
Then, by applying \cref{prop:template} we have:
\begin{align}
\sum_{\run=1}^{\nRuns} \curr[\step]\braket{\nabla \obj(\overline{\state}_{\run+1/2})}{\state_{\run+1/2}-\point}
	&\leq \frac{\hreg(\point)-\min_{\point\in\points}\hreg(\point)}{\learn_{\nRuns+1}}
	\\
	&+ \sum_{\run=1}^{\nRuns}\curr[\step]\braket{\nabla \obj(\overline{\state}_{\run+1/2})-\nabla \obj(\overline{\state}_{\run+1/2})}{\state_{\run}-\next[\state]}
	\notag\\
	&-\sum_{\run=1}^{\nRuns}\frac{\hstr}{2\learn_{\run}}\norm{\next[\state]-\state_{\run+1/2}}^{2}-\sum_{\run=1}^{\nRuns}\frac{\hstr}{2\learn_{\run}}\norm{\state_{\run+1/2}-\state_{\run}}^{2}
\end{align}
Now for the term $\sum_{\run=1}^{\nRuns}\curr[\step]\braket{\nabla \obj(\overline{\state}_{\run+1/2})-\nabla \obj(\overline{\state}_{\run})}{\state_{\run+1/2}-\next[\state]}-\sum_{\run=1}^{\nRuns}\frac{\hstr}{2\learn_{\run}}\norm{\next[\state]-\state_{\run+1/2}}^{2}$ we have:
\begin{align}
- \sum_{\run=1}^{\nRuns}\frac{\hstr}{2\learn_{\run}}\norm{\next[\state]-\state_{\run+1/2}}^{2}
	&+ \sum_{\run=1}^{\nRuns}\curr[\step]\braket{\nabla \obj(\overline{\state}_{\run+1/2})-\nabla \obj(\overline{\state}_{\run+1/2})}{\state_{\run}-\next[\state]}
	&
	\notag\\
	&\leq \frac{1}{2\hstr}\sum_{\run=1}^{\nRuns}\curr[\step]^{2}\learn_{\run}
	\dnorm{\nabla \obj(\overline{\state}_{\runalt+1/2})-\nabla \obj(\overline{\state}_{\runalt})}^{2}
	\notag\\
	&+\frac{\hstr}{2}\sum_{\run=1}^{\nRuns}\frac{1}{\learn_{\run}}\norm{\next[\state]-\state_{\run+1/2}}^{2}-\sum_{\run=1}^{\nRuns}\frac{\hstr}{2\learn_{\run}}\norm{\next[\state]-\state_{\run+1/2}}^{2}
\end{align}
which readily yields:
\begin{align}
\sum_{\run=1}^{\nRuns}\curr[\step]\braket{
\nabla \obj(\overline{\state}_{\run+1/2})-\nabla \obj(\overline{\state}_{\run})}{\state_{\run+1/2}-\next[\state]}
	&- \sum_{\run=1}^{\nRuns}\frac{\hstr}{2\learn_{\run}}\norm{\next[\state]-\state_{\run+1/2}}^{2}
	\notag\\
	&\leq \frac{1}{2\hstr}\sum_{\run=1}^{\nRuns}\curr[\step]^{2}\learn_{\run}
	\dnorm{\nabla \obj(\overline{\state}_{\runalt+1/2})-\nabla \obj(\overline{\state}_{\runalt})}^{2}
\end{align}
Hence, putting everything together, we get:
\begin{align}
\sum_{\run=1}^{\nRuns}\curr[\step]\braket{\nabla \obj(\overline{{\state}}_{{\run+1/2}})}{\state_{\run+1/2}-\point}
	&\leq \frac{1}{2\hstr}\sum_{\run=1}^{\nRuns}\curr[\step]^{2}\learn_{\run}
	\dnorm{\nabla \obj(\overline{\state}_{\runalt+1/2})-\nabla \obj(\overline{\state}_{\runalt})}^{2}
	\notag\\
	&- \frac{\hstr}{2}\sum_{\run=1}^{\nRuns}\frac{1}{\learn_{\run}}\norm{\state_{\run}-\state_{\run+1/2}}^{2} 
	\label{eq:above}
\end{align}

Moreover, since $\obj$ is smooth we have:
\begin{align}
\dnorm{\nabla \obj(\overline{\state}_{\run+1/2})-\nabla \obj(\overline{\state}_{\run})}^{2}
	&\leq \lips^{2} \norm{\overline{\state}_{\run+1/2}-\overline{\state}_{\run}}^{2}
	\notag\\
	&\leq \lips^{2} \frac{\curr[\step]^{2}}{\parens*{\sum_{\run=1}^{\nRuns}\curr[\step]}^{2}}\norm{\state_{\run+1/2}-\state_{\run}}^{2}
	\notag\\
	&=\lips^{2}\frac{4\run^{2}}{\run^{2}(\run+1)^{2}}\norm{\state_{\run+1/2}-\state_{\run}}^{2}
	\notag\\
	&\leq \frac{4\lips^{2}}{\curr[\step]^{2}}\norm{\state_{\run+1/2}-\state_{\run}}^{2}
\end{align}
Combining this with the fact that $\curr[\learn]$ is a decreasing sequence, we can rewrite \eqref{eq:above} as follows:
\begin{align}
\sum_{\run=1}^{\nRuns}\curr[\step]\braket{\nabla \obj(\overline{{\state}}_{{\run+1/2}})}{\state_{\run+1/2}-\point}
	&\leq \frac{\init[\learn]}{2\hstr}\sum_{\run=1}^{\nRuns}\curr[\step]^{2}  \dnorm{\nabla \obj(\overline{\state}_{\runalt+1/2})-\nabla \obj(\overline{\state}_{\runalt})}^{2}
	\notag\\
	&-\frac{\hstr}{8\lips^{2}}\sum_{\run=1}^{\nRuns}\frac{\curr[\step]^{2}}{\learn_{\run}}\dnorm{\nabla \obj(\overline{\state}_{\run+1/2})-\nabla \obj(\overline{\state}_{\run})}^{2}]
	\label{eq:sum}
\end{align}

In the sequel, we look for the appropriate bounds of the two terms in the right-hand-side of \eqref{eq:sum}. We start with the second term. From \eqref{eq:app_adaEW_LemKavis}, we also have $\sum_{\run=1}^{\nRuns}\curr[\step]\braket{\nabla \obj(\overline{{\state}}_{{\run+1/2}})}{\state_{\run+1/2}-\point} \ge 0$. 
%
%
Combine this with \eqref{eq:sum}, we have:
%
%
%
\begin{equation}
	0\leq \frac{1}{2\hstr}\sum_{\run=1}^{\nRuns}\curr[\step]^{2}\learn_{\run}\dnorm{\nabla \obj(\overline{\state}_{\runalt+1/2})-\nabla \obj(\overline{\state}_{\runalt})}^{2}-\frac{\hstr}{8\lips^{2}}\sum_{\run=1}^{\nRuns}\frac{\curr[\step]^{2}}{\learn_{\run}}\dnorm{\nabla \obj(\overline{\state}_{\run+1/2})-\nabla \obj(\overline{\state}_{\run})}^{2}
\end{equation}
Hence by rearranging we have:
\begin{align}
\sum_{\run=1}^{\nRuns} \frac{\hstr\curr[\step]^{2}}{16\learn_{\run}\lips^{2}}\dnorm{\nabla \obj(\overline{\state}_{\run+1/2})-\nabla \obj(\overline{\state}_{\run})}^{2}&\leq \frac{1}{2\hstr}\sum_{\run=1}^{\nRuns}\curr[\step]^{2}\learn_{\run}\dnorm{\nabla \obj(\overline{\state}_{\runalt+1/2})-\nabla \obj(\overline{\state}_{\runalt})}^{2}
	\notag\\
	&- \frac{\hstr}{16\lips^{2}}\sum_{\run=1}^{\nRuns}\frac{\curr[\step]^{2}}{\learn_{\run}}\dnorm{\nabla \obj(\overline{\state}_{\run+1/2})-\nabla \obj(\overline{\state}_{\run})}^{2}
\notag\\
&=\frac{1}{2}\sum_{\run=1}^{\nRuns}\curr[\step]^{2}\dnorm{\nabla \obj(\overline{\state}_{\run+1/2})-\nabla \obj(\overline{\state}_{\run})}^{2}
\left[\frac{\learn_{\run}}{\hstr}-\frac{\hstr}{8\learn_{\run}\lips^{2}}
\right]
\end{align}
Now, since we assumed that $\learn_{\run}$ converges to $0$, there exists some $\run_{0}\in \N$ such that:
\begin{equation}
	\learn_{\run}\leq \frac{\hstr}{\sqrt{8}\lips}\;\;\text{for all}\;\;\run >\run_{0}
\end{equation}
which directly yields that $\left[\frac{\learn_{\run}}{\hstr}-\frac{\hstr}{8\learn_{\run}\lips^{2}}
\right]\leq 0$ for all $\run > \nRuns_0$.
Hence, we have:
\begin{equation}
 \frac{\hstr}{16\lips^{2}}\sum_{\run=1}^{\nRuns}\frac{\curr[\step]^{2}}{\learn_{\run}}\dnorm{\nabla \obj(\overline{\state}_{\run+1/2})-\nabla \obj(\overline{\state}_{\run})}^{2}\leq \frac{1}{2}\sum_{\run=1}^{\nRuns_0}\curr[\step]^{2}\dnorm{\nabla \obj(\overline{\state}_{\run+1/2})-\nabla \obj(\overline{\state}_{\run})}^{2}
\left[\frac{\learn_{\run}}{\hstr}-\frac{\hstr}{8\learn_{\run}\lips^{2}}
\right]
\end{equation}
On the other hand, we have:
\begin{align}
	\frac{1}{\learn_{\run}}&=\frac{1}{\sqrt{\hstr}}\sqrt{\hstr+\sum_{\runalt=1}^{\run-1}\alpha_{\runalt}^{2}\dnorm{\nabla \obj(\overline{\state}_{\runalt+1/2})-\nabla \obj(\overline{\state}_{\runalt})}^{2}}	
	\geq \frac{1}{\sqrt{\hstr}}\sqrt{\hstr}
	= 1
\end{align}
and hence
\begin{align}
\frac{\hstr}{16\lips^{2}}\sum_{\run=1}^{\nRuns}\frac{\curr[\step]^{2}}{\learn_{\run}}\dnorm{\nabla \obj(\overline{\state}_{\run+1/2})-\nabla \obj(\overline{\state}_{\run})}^{2}&\geq \frac{\hstr}{16\lips^{2}}\sum_{\run=1}^{\nRuns}\curr[\step]^{2}\dnorm{\nabla \obj(\overline{\state}_{\run+1/2})-\nabla \obj(\overline{\state}_{\run})}^{2}
	\notag\\
	&=\frac{\hstr^{2}}{16\lips^{2}\hstr}\sum_{\run=1}^{\nRuns}\curr[\step]^{2}\dnorm{\nabla \obj(\overline{\state}_{\run+1/2})-\nabla \obj(\overline{\state}_{\run})}^{2}
	\notag\\
	&=\frac{\hstr^{2}}{16\learn_{\nRuns-1}^{2}\lips^{2}}
\end{align}
So, summarizing we have:
\begin{equation}
\frac{\hstr^{2}}{16\learn_{\nRuns-1}^{2}\lips^{2}}\leq \frac{1}{2}\sum_{\run=1}^{\nRuns_0}\curr[\step]^{2}\dnorm{\nabla \obj(\overline{\state}_{\run+1/2})-\nabla \obj(\overline{\state}_{\run})}^{2}
\left[\frac{\learn_{\run}}{\hstr}-\frac{\hstr}{8\learn_{\run}\lips^{2}}
\right]	
\end{equation}
We now focus on the first term of the right-hand-size of \eqref{eq:sum}. If one lets $\nRuns$ to infinity and recalling the fact that we assumed that $\learn_{\run}$ converges to $0$ (and so $1/\learn_{\run}^{2}\to \infty$), we have that:
\begin{equation}
\infty
	\leq \frac{1}{2}\sum_{\run=1}^{\nRuns_0}\curr[\step]^{2}\dnorm{\nabla \obj(\overline{\state}_{\run+1/2})-\nabla \obj(\overline{\state}_{\run})}^{2}
\left[\frac{\learn_{\run}}{\hstr}-\frac{\hstr}{8\learn_{\run}\lips^{2}}
\right]
	<\infty
\end{equation}
a contradiction.
This shows that $\inf_{\run \in \N}\learn_{\run}>0$ and hence
\begin{align}
	\sum_{\run=1}^{+\infty}\curr[\step]^{2}\dnorm{\nabla \obj(\overline{\state}_{\run+1/2})-\nabla \obj(\overline{\state}_{\run})}^{2}&=\lim_{\nRuns \to +\infty}\sum_{\run=1}^{\nRuns}\curr[\step]^{2}\dnorm{\nabla \obj(\overline{\state}_{\run+1/2})-\nabla \obj(\overline{\state}_{\run})}^{2} \notag
	\\
	&=\lim_{\nRuns \to \infty}\frac{\hstr}{\learn_{\run}^{2}}-\hstr \notag
	\\
	&=\frac{\hstr}{\inf_{\run}\learn_{\run}}^{2} - \hstr
	<\infty
\end{align}
so our proof is complete.
\end{proof}

\section{Additional Numerical Experiments}
\label{app:numerics}

\begin{figure}[t]
	\centering
		\begin{tikzpicture}
			\node (img){\includegraphics[width = .6\textwidth]{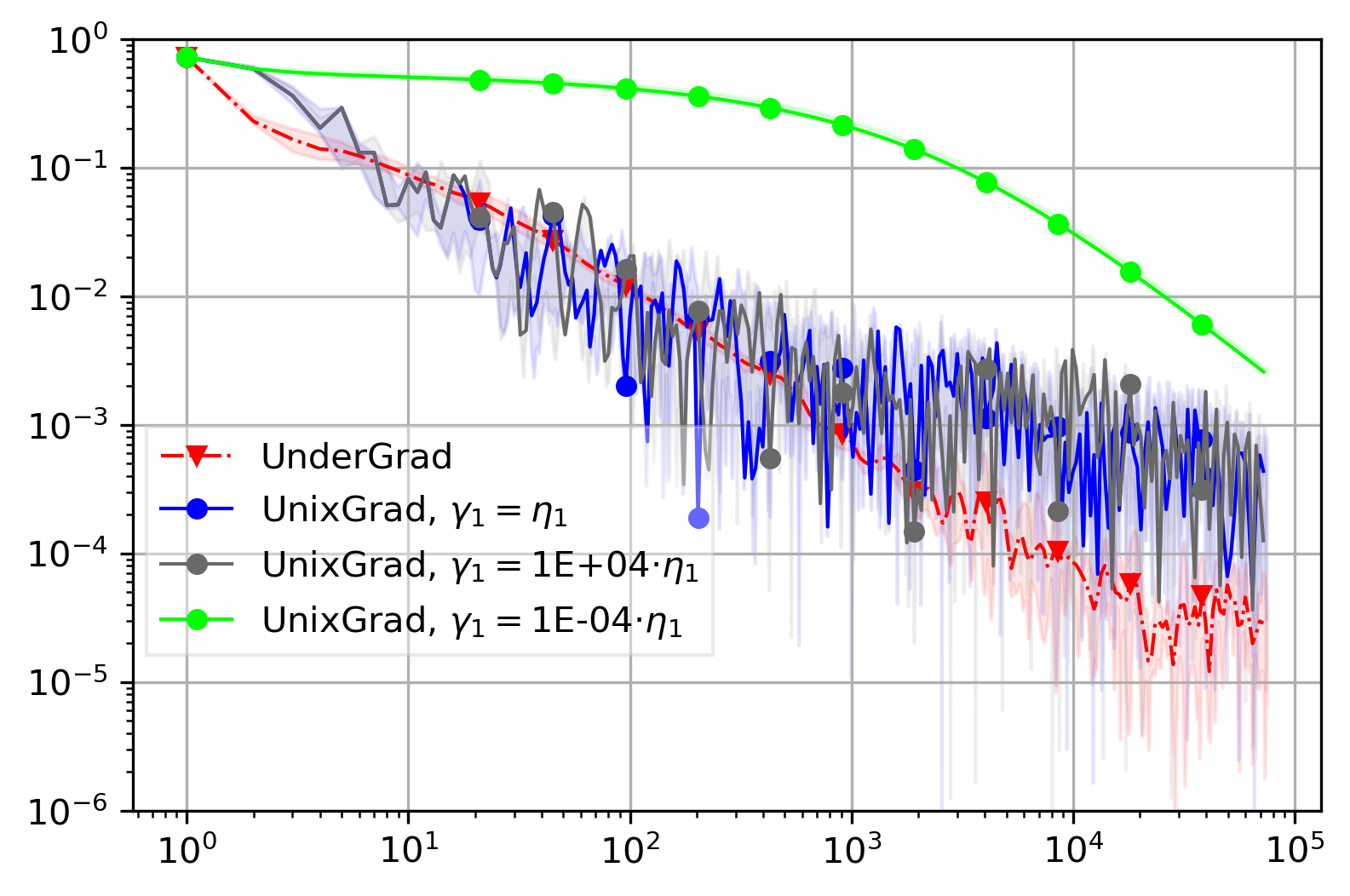}};
			\node[below=of img, node distance=0cm, yshift=1.2cm, xshift =0.2cm] {\scriptsize${\nRuns}$};
			\node[left=of img, node distance=0cm, rotate=90, anchor=center,xshift=0cm, yshift=-1.1cm] {\scriptsize${\gap(\nRuns)}$};
		\end{tikzpicture}
		\caption{\footnotesize Convergence of \undergrad and \unixgrad in the simplex setup with a noisy \ac{SFO}. The plot is drawn in log-log scale. The $y$-axis corresponds to the differences between the $\obj$-value of the relevant point of each algorithm and $\min \obj$.}
		\label{fig:stochastic}
\end{figure}

In this last section, we report another numerical experiment highlighting the universality of \undergrad. In this experiment, we also focus on the simplex setup as presented in \cref{sec:numerics}. However, this time, we work with a noisy \ac{SFO} that returns first-order feedback that is perturbed by a noise generated from a pre-determined zero-mean normal distribution. We compare the performances of \undergrad and \unixgrad, both run with the entropic regularizer. The result of this experiment is reported in \cref{fig:stochastic}.

\cref{fig:stochastic} shows that \undergrad obtains the optimal rate $\bigoh(1/\sqrt{\nRuns})$ in this set-up. \unixgrad can also obtain the same rate but only when its step-size update rule is chosen appropriately (note again that with entropic regularizer, the update rule \eqref{eq:step-unix} of \unixgrad is not available due to the fact that $\bregdiam = \infty$): when $\stepalt_{1}$ is chosen with the same or larger magnitude of \undergrad's initial learning rate, \unixgrad converges with the rate $\bigoh(1/\sqrt{\nRuns})$; but if $\stepalt_1$ is too small (e.g., when $\stepalt_1 = 10^{-3} \cdot \learn_1$), \unixgrad can have a very long warming up phase. This experiment reasserts that in cases where the step-size update rule \eqref{eq:step-unix} is unavailable, it is non-trivial to choose an appropriate step-size update rule of \unixgrad: small $\init[\stepalt]$ might lead to better performances under perfect \ac{SFO} (cf. \cref{sec:numerics}) but might create unwanted behaviors in noisy \ac{SFO} setups. On the contrary, \undergrad does not encounter such issues in our experiments.

Finally,  we conduct another experiment to confirm the dependency of the convergence rates of \undergrad on the noise level $\sdev$. The result is reported in \cref{fig:noise}.
\begin{figure}[t]
	\centering
		\begin{tikzpicture}
			\node (img){\includegraphics[width = .6\textwidth]{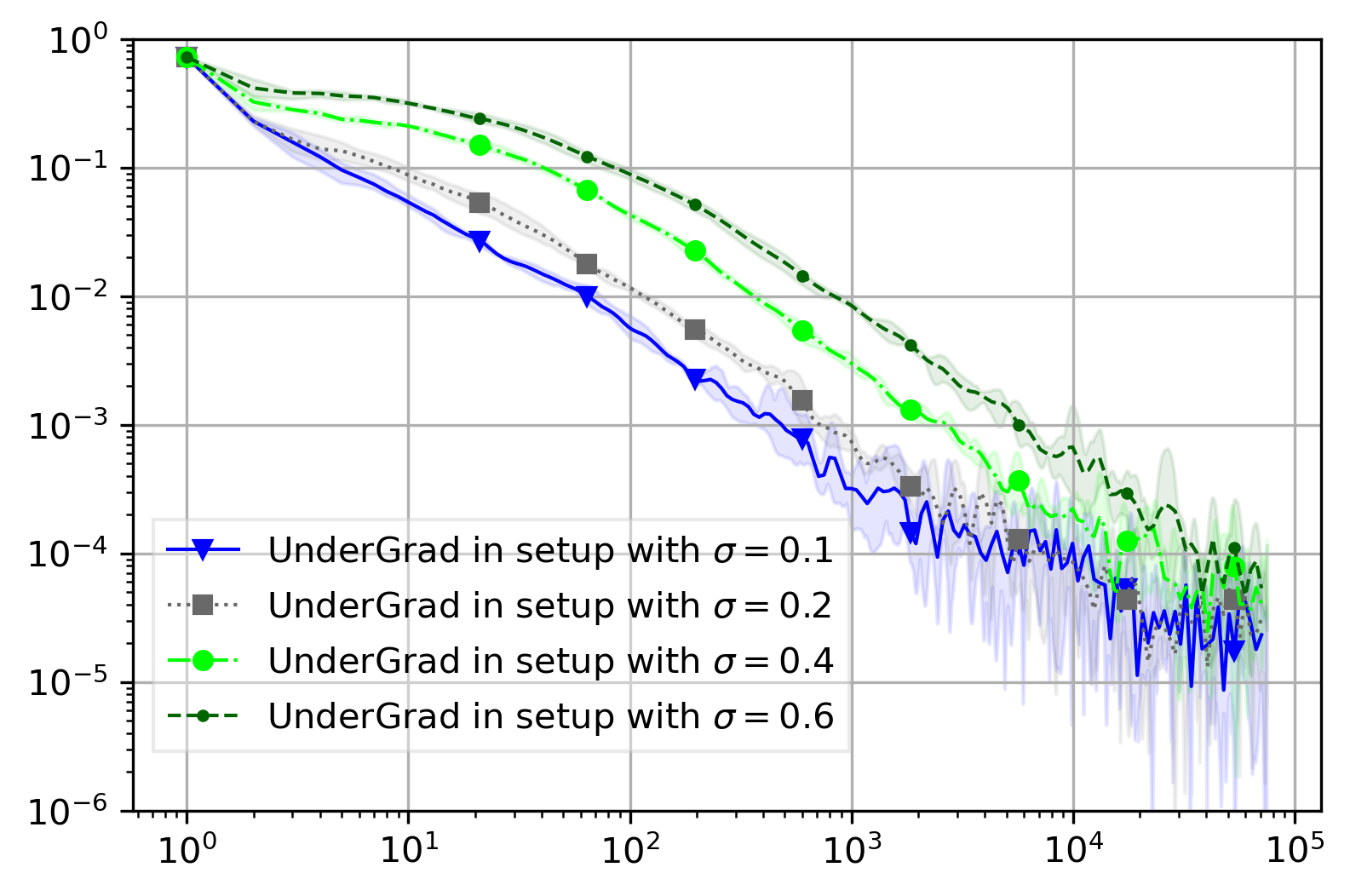}};
			\node[below=of img, node distance=0cm, yshift=1.2cm, xshift =0.2cm] {\scriptsize${\nRuns}$};
			\node[left=of img, node distance=0cm, rotate=90, anchor=center,xshift=0cm, yshift=-1.1cm] {\scriptsize${\gap(\nRuns)}$};
		\end{tikzpicture}
		\caption{\footnotesize Convergence of \undergrad in the simplex setup with different noise levels of the \ac{SFO}.}
		\label{fig:noise}
\end{figure}

\section*{Acknowledgments}
\begingroup
\small
KA and VC were supported by the European Research Council (ERC) under the European Union's Horizon 2020 research and innovation programme (grant agreement \textnumero $725594$ - time-data) and from the Swiss National Science Foundation (SNSF) under  grant number $200021\_205011$.
KYL is grateful for financial support by Israel Science Foundation (grant No. 447/20),  by
Israel PBC-VATAT, and by the Technion Center for Machine Learning and Intelligent Systems (MLIS).  
PM is grateful for financial support by
the French National Research Agency (ANR) in the framework of
the ``Investissements d'avenir'' program (ANR-15-IDEX-02),
the LabEx PERSYVAL (ANR-11-LABX-0025-01),
MIAI@Grenoble Alpes (ANR-19-P3IA-0003),
and
the bilateral ANR-NRF grant ALIAS (ANR-19-CE48-0018-01).
\endgroup

\bibliographystyle{icml2022}
\bibliography{bibtex/IEEEabrv,bibtex/Bibliography-DQV,bibtex/Bibliography-PM}

\end{document}